\documentclass[11pt,oneside,leqno]{amsart}

\makeatletter
\@namedef{subjclassname@2020}{\textup{2020} Mathematics Subject Classification}
\makeatother

\usepackage[top=2.5 cm,bottom=2 cm,left=2 cm,right=3 cm]{geometry}
\usepackage[utf8]{inputenc}
  \usepackage{color}
\usepackage{amssymb}

\usepackage{esint}
\usepackage{graphicx}
\usepackage{hyperref}
\usepackage{bm}
\usepackage{xcolor}

\numberwithin{equation}{section}

\newtheorem{example}{Example}[section]

\newtheorem{theorem}[example]{Theorem}
 \newtheorem{proposition}[example]{Proposition}
\newtheorem{lemma}[example]{Lemma}
 \newtheorem{corollary}[example]{Corollary}
\newtheorem{remark}[example]{Remark}
\newtheorem*{maintheorem*}{Main Theorem}
\allowdisplaybreaks
\numberwithin{equation}{section}

\renewcommand{\i}{\ifmmode\mathit{\mathchar"7010 }\else\char"10 \fi}
\renewcommand{\j}{\ifmmode\mathit{\mathchar"7011 }\else\char"11 \fi}
\newcommand{\R}{\mathbb{R}}
\newcommand{\N}{\mathbb{N}}
\newcommand{\Z}{\mathbb{Z}}
\DeclareMathOperator*{\sign}{sign}
\newcommand{\sgn}[1]{\sign\left(#1\right)}
\newcommand{\abs}[1]{\left|#1\right|}
\newcommand{\norm}[1]{\left\|#1\right\|}
\newcommand{\test}{\varphi}
\newcommand{\seq}[1]{\left\{#1\right\}}
\newcommand{\Dz}{{\Delta z}}
\newcommand{\Dx}{{\Delta x}}
\newcommand{\Dt}{{\Delta t}}
\newcommand{\px}{\partial_x}
\newcommand{\pt}{\partial_t}

\newcommand{\pz}{\partial_z}

\newcommand{\zjp}{z_{j+1/2}}
\newcommand{\zjm}{z_{j-1/2}}
\newcommand{\eps}{\varepsilon}
\newcommand{\ove}[1]{\overline{#1}\negthinspace}
\newcommand{\oW}{\ove{W}}
\newcommand{\oWn}{\oW^{\;n}}
\newcommand{\lenloc}{L^1_{\mathrm{loc}}}

\newcommand{\tx}{\tilde{x}}
\newcommand{\bx}{\bar{x}}
\newcommand{\tu}{\tilde{u}}
\newcommand{\bu}{\bar{u}}

{%

\begin{enumerate}}%
{\end{enumerate}}

{%

\begin{enumerate}}%
{\end{enumerate}}

\begin{document}

\title[Nonlocal Lagrangian traffic flow] {A nonlocal Lagrangian
  traffic flow model and the zero-filter limit}

\author[Coclite]{G. M. Coclite} \author[Karlsen]{K. H. Karlsen}
\author[Risebro]{N. H. Risebro}

\address[Giuseppe Maria Coclite]{\newline Dipartimento di Meccanica,
  Matematica e Management, \newline Politecnico di Bari, \newline Via
  E. Orabona 4 -- 70125 Bari, Italy}
\email[]{giuseppemaria.coclite@poliba.it}

\address[Kenneth Hvistendahl Karlsen, Nils Henrik Risebro]{\newline
  Department of Mathematics, \newline University of Oslo, \newline
  PO Box 1053, Blindern -- 0316 Oslo, Norway}
\email[]{kennethk@math.uio.no, nilshr@math.uio.no}



\thanks{GMC is member of Gruppo Nazionale per l'Analisi Matematica, 
la Probabilit\`a e le loro Applicazioni (GNAMPA), 
which is part of the Istituto Nazionale di Alta Matematica (INdAM). 
GMC has received partial support from the Italian Ministry of 
Education, University and Research (MIUR) through 
the Research Project of National Relevance 
``Evolution problems involving interacting scales" 
(Prin 2022, project code 2022M9BKBC) 
and the Programme Department of Excellence 
Legge 232/2016 (Grant No.~CUP - D94I18000260001). 
GMC expresses his gratitude to the Department of 
Mathematics at the University of Oslo for their warm hospitality. 
This work was partially supported by the project Pure 
Mathematics in Norway, funded by Trond Mohn Foundation 
and Tromsø Research Foundation.}

\date{\today}

\begin{abstract}
  In this study, we start from a Follow-the-Leaders model for traffic
  flow that is based on a weighted harmonic mean (in Lagrangian
  coordinates) of the downstream car density.  This results in a
  nonlocal Lagrangian partial differential equation (PDE) model for
  traffic flow. We demonstrate the well-posedness of 
  the Lagrangian model in the $L^1$ sense. Additionally, we rigorously
  show that our model coincides with the Lagrangian formulation of the
  local LWR model in the ``zero-filter'' (nonlocal-to-local) limit.  We
  present numerical simulations of the new model.  One significant
  advantage of the proposed model is that it allows for simple proofs
  of (i) estimates that do not depend on the ``filter size'' and (ii)
  the dissipation of an arbitrary convex entropy.
\end{abstract}

\maketitle

\tableofcontents

\section{Introduction}

The LWR model, developed by Lighthill, Whitham, and Richards \cite{Lighthill:1955aa} 
more than six decades ago, was the first macroscopic traffic model. The basic
form of the LWR model is a hyperbolic conservation law
\cite{Dafermos:2010fk}, which is a PDE that states that the total
number of vehicles on a given stretch of road must remain constant
over time.  This is expressed mathematically as a continuity equation,
which relates the flow of vehicles $u V$ into and out of a given
region to the change in the density $u$ of vehicles within that
region.  The LWR model also includes equations that describe how the
speed $V$ of vehicles changes over time $t$ and space $x$.  These
equations are based on the assumption that the speed $V$ of a vehicle
located at a point $x$ at time $t$ is determined by the density $u$ of
vehicles at $(t,x)$, $V=V(u(t,x))$, and that the speed of a vehicle
will tend to decrease as the density of surrounding vehicles
increases, $V'(\cdot)\leq 0$.  We refer to $u V(u)$ as the flux
function and the conservation law
\begin{equation}\label{eq:LWR-eqn}
  \partial_t u+\partial_x 
  \bigl(u V(u)\bigr)=0
\end{equation} 
as the original LWR model. There have been many generalisations of the
LWR model over the years. For a comprehensive discussion of traffic
flow data and the various models used to mathematically represent it,
we recommend consulting the book \cite{TreiberKesting:13}.

The original LWR model is based on local PDEs, which means that the
speed function $V$ is determined by the values of the car density at a
single point $x$ in space. There have been numerous efforts to
develop alternative speed functions.  In particular, many authors
examined nonlocal generalisations of the original LWR model, taking
into account the look-ahead distance of drivers in order to better
model their behavior. Some models assume that drivers react to the
mean downstream traffic density, while others assume that they react
to the mean downstream velocity. The corresponding nonlocal LWR models
take the form
\begin{equation}\label{eq:nonlocalCL-intro}
  \partial_t u+\partial_x 
  \left(u V\left(\overline{u}\right)\right)=0,
  \quad 
  \partial_t u+\partial_x 
  \left(u \overline{V(u)}\right)=0,
\end{equation}
where, for a given integrable function $v=v(x)$,
$\overline{v}(x):=\int_x^\infty \Phi_{\alpha}(y-x)v(y)\,dy$.  The
anisotropic kernel $\Phi_{\alpha}$ characterizes the nonlocal effect
through the ``filter size'' $\alpha>0$.  It is a nonnegative,
non-increasing, and $C^1$ function defined on the nonnegative real
numbers, and it has unit mass:
$\int_0^\infty \Phi_{\alpha}(x)\, dx =1$.  Setting
$J_\alpha(x):=\Phi_\alpha(-x)\chi_{(-\infty,0]}(x)$, the
function $\overline{v}(x)$ can be expressed as the convolution
$v\star J_\alpha(x)$, noting that
$\left\{J_\alpha(x)\right\}_{\alpha>0}$ is an approximate identify
(convolution kernel) that generally is discontinuous at $x=0$.  In the
formal limit $\alpha\to 0$ (the ``zero-filter'' limit), the nonlocal
fluxes $u V\left(\overline{u}\right)$ and $u \overline{V(u)}$
converge to the local flux $u V(u)$ of the original LWR model
\eqref{eq:LWR-eqn}.

The mathematical study of conservation laws with nonlocal flux has
gained significant attention in recent years.  A comprehensive list of
references on this topic is beyond the scope of this text.  Instead,
we refer the reader to the recent paper \cite{CocliteEtal:22} (on weak
solutions) and the references cited therein. Here we only mention a
few references
\cite{BlandinGoatin:16,ChiarelloGoatin:18,FriedrichEtal:18,GoatinScialanga:16}
related to nonlocal conservation laws \eqref{eq:nonlocalCL-intro} that
arise as generalisations of the original LWR model.  In particular, in
\cite{BlandinGoatin:16,ChiarelloGoatin:18,GoatinScialanga:16} the
authors establish the well-posedness (of entropy solutions) and
convergence of numerical schemes for the first equation in
\eqref{eq:nonlocalCL-intro}, as well as a more general version of it.
For modifications of these results to account for the second equation
in \eqref{eq:nonlocalCL-intro}, see \cite{FriedrichEtal:18}.

In general \cite{ColomboEtal:21}, solutions of nonlocal conservation
laws like $\partial_t u_\alpha +\partial_x \bigl(u_\alpha V(u_\alpha\star
J_\alpha)\bigr)=0$, where $J_\alpha$ is an arbitrary approximate
identity and $V$ is a Lipschitz function, do not converge to the
entropy solution of the corresponding local conservation law as the
``filter size" $\alpha$ approaches zero. The counterexamples in
\cite{ColomboEtal:21} do not exclude the possibility that convergence
may still hold in specific cases.  In particular, the case where
$V'(\cdot)\leq 0$, the initial function is nonnegative, and the
convolution kernel $J_\alpha$ is anisotropic, specifically supported
on the negative axis $(-\infty, 0]$. This case corresponds to nonlocal
traffic flow PDEs, like the first one in \eqref{eq:nonlocalCL-intro}.
Recently, under assumptions like these, positive results have been
obtained for the zero-filter limit
\cite{BressanShen:20,BressanShen:21,CocliteLimit:20,Colombo:22,Keimer:19}.

Traffic flow models can be divided into two categories: macroscopic
models, which describe the flow of vehicles on a roadway as a
continuous fluid, and microscopic models, which describe the motion
and interactions of individual vehicles.  LWR-type PDEs are examples
of macroscopic traffic flow models, while microscopic models are often
described using systems of differential equations, such as the
Follow-the-Leaders (FtL) model. In the FtL model, the velocity of each
vehicle is determined by the velocity of the vehicle in front of
it. There is a (rigorous) connection between FtL models and hyperbolic
conservation laws, which has been studied in detail in the literature,
see \cite{DiFrancesco:15,HoldenRisebro:18} and the references
therein. In
\cite{Chiarello:20,ChienShen:2019,RidderShen:19,ShenShikh-Khalil:18},
the authors provide links between nonlocal FtL models and the
macroscopic LWR-type equations \eqref{eq:nonlocalCL-intro}.

Before we present our own model, it is helpful to briefly describe the
nonlocal FtL models of \cite{ChienShen:2019,
  RidderShen:19,ShenShikh-Khalil:18}.  Let $x_i(t)$, $i\in \Z$, be the
position of the $i$th car, ordering them so that
$x_{i+1}(t)\ge x_i(t)+\ell$, where $\ell$ is the (common) length of
the cars. Set
\begin{equation}\label{eq:uidef}
u_i(t):= \frac{\ell}{x_{i+t}(t)-x_i(t)},
\end{equation}
which is the local discrete density (or ``car saturation") perceived
by the driver of car $i\in \Bbb{Z}$.  One of the nonlocal FtL models
of \cite{RidderShen:19} asks that the car positions $x_i(t)$ satisfy
the following system of differential equations:
\begin{equation}\label{eq:nonlocal-FtL-density}
  x_i'(t)=V\left(\overline{u_i}(t)\right), 
  \quad
  i \in \Z, \,\, t>0,
\end{equation}
where
\begin{equation}\label{eq:arithmetic-mean}
  \overline{u_i}(t)
  :=\sum_{j=0}^\infty \Phi_{ij\alpha}(t)
  u_{i+j}(t),\quad 
  \Phi_{ij\alpha}(t):=\int_{x_{i+j}(t)}^{x_{i+1+j}(t)}
  \Phi_{\alpha}(\zeta-x_i(t))\, d\zeta, 
  \quad i\in \Z.
\end{equation}
In other words, the velocity of each vehicle is not only determined by
the vehicle directly in front of it, but also by the other vehicles in
the surrounding (downstream) area.  Replacing
\eqref{eq:nonlocal-FtL-density} by
\begin{equation}\label{eq:nonlocal-FtL-velocity}
  x_i'(t)=\overline{V(u_i(t))}, 
  \quad
  i \in \Z, \,\, t>0,
\end{equation}
we obtain a slightly different FtL model.  While drivers under the
model \eqref{eq:nonlocal-FtL-density} react to the mean downstream
traffic saturation, drivers under the model
\eqref{eq:nonlocal-FtL-velocity} react to the mean downstream
velocity.

The nonlocal FtL model \eqref{eq:nonlocal-FtL-density},
\eqref{eq:arithmetic-mean} uses a weighted \textit{arithmetic mean} of
the (downstream) car-density values to calculate the speed. There are
several ways to aggregate a sequence of numbers.  While the arithmetic
mean is a simple average calculated by adding up the values in a set
and dividing by the number of values, the harmonic mean is calculated
by taking the reciprocal of the arithmetic mean of the reciprocals of
the values in a set. In view of the well-known harmonic
mean-arithmetic mean inequality \cite[p.~126]{Steele-MasterClass:04},
the harmonic mean is generally a more conservative estimate of the
average value in a set; roughly speaking, the harmonic mean takes into
account the ``size'' of the values in the set, while the arithmetic
mean does not.

In this paper we propose a nonlocal FtL model based on
a weighted harmonic mean in the Lagrangian coordinates. The governing
differential equations are of the form
\begin{equation}\label{eq:nonlocal-FtL-harmonic}
  x_i'(t)=
  V\left(\,\left[\, \overline{\frac{1}{u_i(t)}}
      \,\right]^{-1}\,\right), 
  \quad
    \overline{\frac{1}{u_i(t)}}
  :=\sum_{j=0}^\infty 
  \frac{\Phi_{ij\alpha}}{u_{i+j}(t)}, \qquad i \in \Z, \,\, t>0.
\end{equation}
Now the weights are determined by
\begin{equation}
  \label{eq:Lagrangeweights}
  \Phi_{ij \alpha}:=\int_{z_{i+j}}^{z_{i+1+j}}
  \Phi_{\alpha}(\zeta-z_i)\, d\zeta, 
  \quad j=0,1,2,\ldots,
\end{equation}
where $z_i:=i\ell$ is the Lagrangian coordinate of the $i$-th car.
Note carefully that the weights $\Phi_{ij \alpha}$ are computed by
averaging the kernel $\Phi(\cdot-z_i)$ (centered at car $i$) between
the Lagrangian particles $z_{i+j}$ (car $i+j$) and $z_{i+1+j}$ (car
$i+1+j$). The cars are here labelled in the driving
direction\footnote{In some Lagrangian traffic models (see, e.g.,
  \cite{Leclercq2007TheLC}), the so-called cumulative count function
  $N(x,t)$ is used, which represents the number of cars that have
  passed a specific location ($x$) at a specific time ($t$), starting
  with a reference car that is labelled as $1$.  As cars pass the
  observer, they are labeled in consecutive order (2, 3, 4, etc.),
  thereby labelling the cars in the opposite direction of their
  driving direction.  By ordering the cars in the driving direction
  (as we do here), the first car would be the one closest to the point
  of observation and the car numbering would increase as the cars move
  further away from the point of observation.  The corresponding
  cumulative count function $\tilde N(x,t)$ then represents the number
  of cars that have yet to pass a certain point in the road at a given
  time. This means that the value of $\tilde N(x,\cdot)$ will decrease
  over time as more cars pass the point of observation $x$, while
  $N(x,\cdot)$ increases.}, so that the weights $\Phi_{ij \alpha}$
decrease with the car number (increasing $z_i$).  Averaging between 
Lagrangian particles is different
from the more traditional approach \eqref{eq:arithmetic-mean}.  The
contrast between the position $x_i$ of car $i$ and the Lagrangian
coordinate $z_i$ is that $x_i$ represents the actual physical position
of the car in space, while $z_i$ is a mathematical construct
(labelling) used to describe the car's position relative to other cars.

The corresponding macroscopic equation becomes
\begin{equation}\label{eq:intro-new-nonlocal-PDE}
  \partial_t \left(\frac{1}{u(z,t)}\right)
  -\partial_z V\left( \,\left[\,
  \overline{\frac{1}{u(z,t)}}\,\right]^{-1}
  \, \right)=0, \quad z\in \R, \,\, t>0, 
\end{equation}  
where
\begin{equation}\label{eq:intro-new-nonlocal-PDE-II}
	\overline{\frac{1}{u(z,t)}}= \int_z^\infty \Phi_{\alpha}(\zeta-z)
	\frac{1}{u(\zeta,t)}\, d\zeta.
\end{equation}
In other words, in terms of the Lagrangian variable
$y=y(z,t)=\frac{1}{u(z,t)}$ (``amount of road per car'', also known as
``spacing'' or ``gap'' between cars), we obtain a nonlocal conservation
law of the form
\begin{equation}\label{eq:intro-new-model}
  \partial_t y-\partial_z W\left(\overline{y}\right)=0, 
  \quad 
  \overline{y}(z,t)=
  \int_z^\infty \Phi_{\alpha}(\zeta-z)y(\zeta,t)\, d\zeta,
  \quad W(y):=V\left(\frac{1}{y}\right).
\end{equation}

Formally, as the filter size $\alpha$ approaches zero, the local
Lagrangian PDE $\partial_t (1/u)-\partial_z V(u)=0$ is obtained. This
PDE can be transformed into the Eulerian PDE \eqref{eq:LWR-eqn}
through a change of variable \cite{Wagner:1987aa}.  The nonlocal LWR
equations \eqref{eq:LWR-eqn} are Eulerian models, while the model
\eqref{eq:intro-new-nonlocal-PDE} analysed in this paper is a
Lagrangian model. The main difference between the two is the
coordinate system used. In Eulerian coordinates, traffic is observed
from a fixed point and the coordinates are fixed in space, while in
Lagrangian coordinates, traffic is observed from a car travelling with
the flow and coordinates move with the cars.  In Eulerian coordinates,
the main variable is density $u$ as a function of space $x$ and time
$t$, while in the Lagrangian formulation, it is spacing $y$ as a
function of ``car number" $z$ and time $t$ (the smaller the spacing,
the higher the traffic density, and vice versa).  Lagrangian traffic
flow models have become increasingly important in recent times, as
advancements in technology have allowed for the collection of data via
GPS, on-board sensors, and smartphones.  This provides more accurate
Lagrangian traffic measurements.

We will see that the mathematical and numerical analysis of the
Lagrangian PDE \eqref{eq:intro-new-nonlocal-PDE} becomes fairly
simple, whereas its Eulerian counterpart leads to a complicated PDE
that appears much harder to analyse directly. Besides, we are able to
rigourously justify the zero-filter limit of
\eqref{eq:intro-new-nonlocal-PDE}. More precisely, we show the
existence, uniqueness, and $L^1$ stability of solutions to
\eqref{eq:intro-new-nonlocal-PDE}, for any fixed value 
of the filter size $\alpha>0$. 
To prove the existence of a weak solution, we use
approximate solutions obtained from the FtL model 
and compactness arguments. The resulting solution 
is regular enough to make it easy to prove the
uniqueness and stability of the weak solution.  A key aspect of our
approach is that we derive estimates and strong convergence for the
filtered variable
\begin{equation}\label{eq:filter-var}
  w:=\overline{y}=\int_z^\infty \Phi_{\alpha}(\zeta-z)
  y(\zeta,t)\, d\zeta,
\end{equation}
rather than for the original variable $y=1/u$ itself.  This allows for
simple proofs of estimates that are independent of the filter size
$\alpha$, which is at variance with the more traditional analyses of
\cite{BlandinGoatin:16,ChiarelloGoatin:18,FriedrichEtal:18,GoatinScialanga:16}.
As a result, we can consider a sequence
$\left\{w_{\alpha}=\overline{y_\alpha}\right\}_{\alpha>0}$ of filtered
solutions of \eqref{eq:intro-new-nonlocal-PDE} and show that a
subsequence converges strongly in $L^1_{\operatorname{loc}}$ to a
function $w$ that is a solution of the (Lagrangian form) of the
LWR equation \eqref{eq:LWR-eqn}. Besides, we demonstrate that
$w_{\alpha}$ dissipates any convex entropy function, which implies
that the limit $w$ is the unique Kru{\v{z}}kov entropy solution of the
LWR equation. We even provide an explicit rate of convergence, namely that
$\norm{w_{\alpha}(t)-w(t)}_{L^1(\R)}\leq C \sqrt{\alpha}$.  It is
worth noting that the zero-filter limit has only recently been
successfully studied in \cite{CocliteLimit:20,Colombo:22}, but only
for the first nonlocal conservation law in
\eqref{eq:nonlocalCL-intro}.  Our work provides a different approach
for studying an alternative nonlocal Langrangian model
\eqref{eq:intro-new-nonlocal-PDE}, which is distinct from
\eqref{eq:nonlocalCL-intro}, and its zero-filter limit.

In this study, we also demonstrate that the variable $y_\alpha$ 
converges strongly through the estimation of $w_\alpha-y_\alpha$ 
in the $L^1$ norm for exponential kernels. Based on numerical experiments, 
the same appears to be true for Lipschitz kernels. However, the 
convergence is not expected for general discontinuous kernels. 
Our numerical experiments indicate that as $\alpha$ approaches zero, 
oscillations persist in the variable $y_\alpha$ 
for discontinuous kernels.

\smallskip

The paper is structured as follows:
Section \ref{sec:conv-scheme} analyzes a fully discrete scheme for $w_\alpha$.
Section \ref{sec:PDE-for-y} explores the connection between 
$y_\alpha=1/u_\alpha$ and $w_\alpha$.
Section \ref{sec:Eulerian} provides an Eulerian formulation
for the discussed Lagrangian PDE for easy comparison with existing literature.
Section \ref{sec:zero-filter} examines the zero-filter limit.
Finally, Section \ref{sec:numerics} showcases numerical examples.

\section{Analysis of a fully discrete scheme}\label{sec:conv-scheme}

In this section, we will present and analyze a fully discrete
numerical approach based on the nonlocal FtL model
\eqref{eq:nonlocal-FtL-harmonic}.  The numerical examples for this
approach will be provided in Section \ref{sec:numerics}. Before that,
however, we will list some properties of the averaging kernel
$\Phi_\alpha$ and the associated averaging operator.
 
Let $\Phi:\R_+\to \R_+$ be a non-increasing function such that
\begin{equation}\label{eq:kernel-ass1}
  \int_0^\infty \Phi(z)\,dz = 1\ \ \ \text{and}\ \ \
  \int_0^\infty z\Phi(z)\,dz < \infty.
\end{equation}
For $\alpha>0$ define
\begin{equation}\label{eq:kernel-ass2}
  \Phi_\alpha(z)=\frac{1}{\alpha}
  \Phi\Bigl(\frac{z}{\alpha}\Bigr),
\end{equation}
and for any suitable function $h:\R\to\R$ define
\begin{equation}\label{eq:average-op}
  \overline{h}(z)=\int_z^\infty \Phi_\alpha(\zeta-z)h(\zeta)\,d\zeta=
  \int_0^\infty \Phi_\alpha(\zeta)h(z+\zeta)\,d\zeta.
\end{equation}
We have that
\begin{align*}
  & \overline{h}'(z)
  =\overline{h'}(z)\ \ 
    \text{if $h$ is differentiable,}\\
  & \norm{\overline{h}}_{L^p(\R)}
  \le \norm{h}_{L^p(\R)},\ \ p\in [1,\infty],\\
  & \overline{h}'(z)
  =\int_0^\infty \Phi'_\alpha(\zeta)
    \left[h(z)-h(z+\zeta)\right]\,d\zeta
    =-\int_z^\infty \Phi'_\alpha(\zeta-z)
    \left[h(\zeta)-h(z)\right]\,d\zeta,   
\end{align*}
if $\Phi$ is differentiable.

We shall consider a time-forward Euler discretization of the system of
ODEs \eqref{eq:nonlocal-FtL-harmonic}. We set $\Dz=\ell>0$ and employ
the usual notation $z_j=(j-1/2)\Dz$, $j \in \Z/2$, $z_{1/2}=0$, and
$\lambda=\Dt/\Dz$, where $\Dt>0$ is a sufficiently small (to be
specified) number.  Subtracting the equation for $x_i'$ in
\eqref{eq:nonlocal-FtL-harmonic} from that for $x_{i+1}'$ and dividing
the result by $\Dz$, we get
\begin{equation}\label{eq:ode1}
  \frac{d}{dt} \Bigl(\frac{1}{u_i(t)}\Bigr) =
  \frac{1}{\Dz} \Bigl( V\Bigl(\, \overline{\frac{1}{u_{i+1}}}\, \Bigr)
  - V\Bigl(\,\overline{\frac{1}{u_{i}}}\,\Bigr)\Bigr),
\end{equation}
where
\begin{equation*}
  \overline{h}_i = \sum_{j\ge i} \Phi_{ij\alpha}h_j, \quad
  \Phi_{ij\alpha}=\int_{z_{j-1/2}}^{z_{j+1/2}} 
  \Phi_\alpha(\zeta-z_{i-1/2})\,d\zeta,
  \quad i\in \Z,
\end{equation*}
and we have used \eqref{eq:uidef}. The semi-discrete 
scheme \eqref{eq:ode1} represents an approximation
of the nonlocal Lagrangian PDE \eqref{eq:intro-new-nonlocal-PDE}. 
Throughout the paper, $\Phi_{ij\alpha}$ and $\Phi_{i,j,\alpha}$ 
are used interchangeably, with either commas or no commas 
in their notation.

To greatly facilitate the analysis, we will shift our focus from the
variable $y=1/u$ to its filtered counterpart by introducing
\begin{equation}\label{eq:wi-average-Wdef}
  w_i=\overline{\frac{1}{u_{i}}}, 
  \quad W(w)=V\Bigl(\frac{1}{w}\Bigr),
  \quad \text{$V\in C^1([0,\infty))$ non-increasing}
\end{equation}
as previously mentioned in the introduction,
cf.~\eqref{eq:filter-var}.

Applying the $\overline{\;\cdot\;}$ operator to \eqref{eq:ode1}, we
get
\begin{equation*}
  \frac{d}{dt} w_i =
  \frac{1}{\Dz}\left(\overline{W(w_{i+1})}
    -\overline{W(w_{i})}\right),
  \ \ \ i\in\Z.  
\end{equation*}
We shall analyse the following scheme for the this system of ODEs:
\begin{equation}
  \label{eq:wscheme}
  \begin{aligned}
    w^{n+1}_i &= w^n_i + \lambda
    \left(\oWn_{i+1/2}-\oWn_{i-1/2}\right),\ \ n\ge 0,\\
    w^0_i &= \sum_{j\ge i} \Phi_{ij\alpha} y_{0,j},
  \end{aligned}
  \ \ i\in \Z,
\end{equation}
where $w^n_i\approx w_i(n\Dt)$ and
\begin{equation*}
  \oWn_{i-1/2}=\sum_{j\ge i} \Phi_{ij\alpha} W(w^n_j),\ \ \ \
  y_{0,i}=\frac{1}{u_{0,i}}=\frac{x_{i+1}(0)-x_i(0)}{\ell}.
\end{equation*}

It is readily verified that the infinite 
matrix $\Phi_{ij\alpha}$ satisfies
\begin{align*}
  \Phi_{i-1,j-1,\alpha}
  &=\int_{z_{j-3/2}}^{\zjm}\Phi_\alpha(\zeta-z_{i-3/2})\,d\zeta
    =\int_{\zjm}^{\zjp} \Phi_\alpha(\zeta-z_{i-1/2})\,d\zeta=\Phi_{ij\alpha},\\
  \sum_{j\ge i} \Phi_{ij\alpha}
  &= \sum_{j\ge 1} \Phi_{1j\alpha} = \sum_{j\ge 1}
    \int_{\zjm}^{\zjp}
    \Phi_\alpha(\zeta)\,d\zeta=\int_0^\infty
    \Phi(\zeta)\,d\zeta = 1,\\
  \sum_{i\in\Z} \sum_{j\ge i} \Phi_{ij\alpha} \, \mu_j
  &=\sum_{i\in\Z} \sum_{k=1}^\infty \Phi_{i,i+k-1,\alpha} \,  \mu_{i+k-1}
    =\sum_{i\in\Z} \sum_{k=1}^\infty \Phi_{1,k,\alpha} \,  \mu_{i+k-1}
    =\sum_{j\in\Z} \mu_j \sum_{k=1}^\infty \Phi_{1,k,\alpha} \\
  &= \sum_{i\in\Z} \mu_i,\\
  \sum_{i\in\Z} \Bigl|\sum_{j\ge i} \Phi_{ij\alpha} \, \mu_j\Bigr|^p
  &\le \sum_{i\in\Z} \abs{\mu_j}^p, \ \ 1\le p<\infty,\\
  \sup_{i\in\Z}\Bigl|\sum_{j\ge i} \Phi_{ij\alpha} \, \mu_j\Bigr|
  &\le \sup_{i\in\Z}\abs{\mu_i}.
\end{align*}

The following lemma demonstrates that the scheme \eqref{eq:wscheme}
for the filtered variable $w=\overline{y}$ adheres to the classical
monotonicity criteria of Harten, Hyman, and Lax.  The monotonicity of
the scheme ensures that the numerical solution does not create
spurious oscillations or produce unphysical values outside of the set
of initial conditions. Note that the (exact) solution operator for the 
original variable $y=1/u$ is not monotone.

\begin{lemma}\label{lem:monotone}
  If $\Dt$ and $\Dx$ are chosen such that the CFL -condition
  \begin{equation}\label{eq:CFLcond}
    0\le \lambda \sup_wW'(w) \le 1
  \end{equation}
  holds, then the scheme \eqref{eq:wscheme} is monotone in the sense
  that 
  \begin{equation*}
    w^n_i \ge \widetilde{w}^n_i \ \ \text{for all $i\in\Z$}  \ \ \
    \Longrightarrow\ \ \
    w^{n+1}_i \ge \widetilde{w}^{n+1}_i \ \ \text{for all $i\in\Z$,}
  \end{equation*}
  where $\widetilde{w}^{n+1}$ is a corresponding 
  solution of \eqref{eq:wscheme}.
\end{lemma}
\begin{proof}
  We compute
  \begin{equation*}
    \frac{\partial w^{n+1}_i}{\partial w^n_k}=
    \begin{cases}
      0, &k<i,\\
      1-\lambda \Phi_{ii\alpha} W'(w_i), &k=i,\\
      \lambda \bigl(\Phi_{ik\alpha} -\Phi_{i,k+1,\alpha}\bigr)
      W'(w_k), &k>i,
    \end{cases}\ \ \ \ge 0,
  \end{equation*}
  if \eqref{eq:CFLcond} holds, since $\Phi_{ii\alpha}\le 1$ and 
  $\Phi_{ik\alpha} -\Phi_{i,k+1,\alpha}\ge 0$.
\end{proof}

As a direct result of the monotonicity, the scheme \eqref{eq:wscheme}
for the filtered variable $w$ is also $L^1$ contractive (stable with
respect to the initial data).

\begin{corollary}\label{cor:l1contr}
  Assume that the CFL-condition \eqref{eq:CFLcond} holds
  and let $\widetilde{w}^n_i$ be the result of applying the scheme
  \eqref{eq:wscheme} to the initial data $\widetilde{y}_{0,i}$. Then
  \begin{equation*}
    \Dz \sum_i\abs{w^n_i-\widetilde{w}^n_i}\le \Dz
    \sum_i\abs{w^0_i-\widetilde{w}^0_i}
    = \Dz \sum_i\abs{y_{0,i}-\tilde{y}_{0,i}}.
  \end{equation*}
\end{corollary}

\begin{proof}
  Since the scheme is monotone, we can use 
  the Crandall-Tartar lemma \cite[Lemma 2.13]{Holden:2015aa} on the set
  \begin{equation*}
    D_{a,b}=\Bigl\{ \seq{w_i}_{i\in\Z} \;\Bigm|\; 1\le w_i<\infty,\  \
    \Dz\sum_{i\le 0} \abs{w_i-a}<\infty,
    \ \ \Dz\sum_{i\ge 0} \abs{w_i-b}<\infty\Bigr\},  
  \end{equation*}
  and conclude that the corollary holds.
\end{proof}

The monotonicity of the scheme \eqref{eq:wscheme} for the filtered
variable implies several basic estimates that are independent of the
filter size $\alpha$. This is a key feature of using the filtered
variable, as it allows for the numerical scheme to be stable and
well-balanced as $\alpha \to 0$. These estimates are not used to prove
the convergence of the scheme to the filtered version of the nonlocal
Lagrangian PDE \eqref{eq:intro-new-nonlocal-PDE} (for fixed $\alpha$),
but rather to address the behavior of the scheme in the limit as
$\alpha$ approaches zero.  This is important because it helps to
ensure consistency with the original LWR model. 
We will return to the zero-filter limit of
\eqref{eq:intro-new-nonlocal-PDE} in Section \ref{sec:zero-filter}.

\begin{corollary}
  \label{cor:bounds}
  Assume that the CFL-condition \eqref{eq:CFLcond} holds. Then
  \begin{align}
    \label{eq:supbnd}
    1\le \inf_{i} y_{0,i}
   & \le  w^n_i\le \sup_{i} y_{0,i},
    \\
    \label{eq:bvbound}
    \sum_i\abs{w^n_{i+1}-w^n_i}
    &\le \sum_i\abs{y_{0,i+1}-y_{0,i}},
    \\
    \label{eq:L1cont}
    \Dz\sum_i \abs{w^{n+1}_i-w^n_i}
    & \le \Dt \norm{W'}_{L^\infty}\abs{y_{0,\cdot}}_{BV}.
  \end{align}
\end{corollary}

\begin{proof}
  To prove \eqref{eq:supbnd}, observe that the constants
  $c=\inf_{i} y_{0,i}$ and $C=\sup_{i} y_{0,i}$ are solutions to the
  scheme \eqref{eq:wscheme} and then apply monotonicity. To prove the $BV$
  bound \eqref{eq:bvbound}, set $\widetilde{w}^n_i=w^n_{i+1}$ in
  Corollary~\ref{cor:l1contr}. To prove $L^1$-continuity
  \eqref{eq:L1cont}, choose $\widetilde{w}=w^{n+1}_i$ in
  Corollary~\ref{cor:l1contr} and calculate
  \begin{align*}
    \Dz \sum_i\abs{w^{n+1}_i-w^n_i}
    &\le \Dz \sum_i\abs{w^1_i-w^0_i}
    = \Dt \sum_i\abs{\oW^{\;0}_{i+1/2}-\oW^{\;0}_{i-1/2}}\\
    &= \Dt \sum_i \Bigl|
      \sum_{j\ge i+1} \Phi_{i+1,j,\alpha}W(w^0_{j})- \sum_{j\ge i}
      \Phi_{ij\alpha}W(w^0_{j})\Bigr|\\
    &=\Dt \sum_i \Bigl|
      \sum_{j\ge i} \Phi_{ij\alpha}W(w^0_{j+1})- \sum_{j\ge i}
      \Phi_{ij\alpha}W(w^0_{j})\Bigr|\\
    &\le \Dt \sum_i \sum_{j\ge i}\Phi_{ij\alpha} \abs{W(w^0_{j+1})-W(w^0_{j})}
    \\
    &\le \Dt \norm{W'}_\infty\sum_i \abs{w^0_{j+1}-w^0_{j}}
      \le \Dt\norm{W'}_\infty \abs{y_{0,\cdot}}_{BV}.
  \end{align*}
\end{proof}

Next, we will estimate the variations in space and time of the
solution $w_i^n$ of the scheme \eqref{eq:wscheme} for the filtered
variable $w=\overline{y}$. These estimates will be dependent on the
filter size $\alpha$, but they will be sufficient to demonstrate
uniform convergence to a Lipschitz continuous limit $w_\alpha(x,t)$
for a fixed value of $\alpha$. As we wish to bound the ``derivatives''
of $w^n_i$, let us define
\begin{equation*}
  \Delta w^n_{j+1/2}=w^n_{j+1}-w^n_j,\ \ 
  \Delta W_j=W(w_{j+1})-W(w_j)\ \ \text{and}\ \ 
  \Delta \oW_j = \oW_{j+1/2}-\oW_{j-1/2},
\end{equation*}
and set
\begin{equation}\label{eq:supw}
  \left(\Delta\widehat{w}\right)^n=\sup_i\abs{\Delta
    w^n_{i+1/2}}.
\end{equation}
Note that $\Delta \oW_i = \sum_{j\ge i} \Phi_{ij\alpha}\Delta W_j$.

\begin{lemma}
  \label{lem:derivs}
  Assume that the CFL-condition \eqref{eq:CFLcond} holds. We have
  \begin{align}
    \label{eq:xderiv}
    (\Delta\widehat{w})^{n}
    &\le (\Delta\hat{w})^0 \exp\Bigl(\frac{C}{\alpha} t^n\Bigr),\\
    \sup_i\abs{w^{n+1}_i - w^n_i} \label{eq:tderiv}
    &\le \lambda \norm{W'}_\infty (\Delta\widehat{w})^0
      \exp\Bigl(\frac{C}{\alpha} t^n\Bigr),
  \end{align}
  where $t^n=n\Dt$, $(\Delta\widehat{w})^{n}$ is defined in
  \eqref{eq:supw}, and the constant $C$ 
  is independent of $n$, $\Dz$ and $\alpha$.
\end{lemma}
\begin{proof}
  We calculate
  \begin{align*}
    \abs{\Delta w^{n+1}_{i+1/2}}
    &=\abs{\Delta w^n_{i+1/2} + \lambda \left(\Delta
      \oWn_{j+1}-\Delta \oWn_{j}\right)}\\
    &\le  \abs{\Delta w^n_{i+1/2}} + \lambda\Bigl|
      \sum_{j\ge i+1} \Phi_{i+1,j,\alpha} \Delta W^n_j - \sum_{j\ge i}
      \Phi_{ij\alpha}\Delta W^n_j\Bigr|\\
    &=\abs{\Delta w^n_{i+1/2}} + \lambda \Bigl|\sum_{j\ge i+1}
      \left(\Phi_{i+1,j,\alpha}-\Phi_{ij\alpha}\right)\Delta W^n_j
      -\lambda\Phi_{1,1,\alpha}\Delta W^n_i\Bigr|\\
    &\le \abs{\Delta w^n_{i+1/2}} -\lambda \sum_{j\ge 1}
      \left(\Phi_{1,j+1,\alpha}-\Phi_{1,j,\alpha}\right)\abs{\Delta W^n_j}
      +\lambda\Phi_{1,1,\alpha}\abs{\Delta W^n_i}\\
    &\le \abs{\Delta w^n_{i+1/2}} + \Dt\norm{W'}_\infty\sum_{j\ge 1}
      \frac{\Phi_{1,j,\alpha}-\Phi_{1,j+1,\alpha}}{\Dz}\abs{\Delta w^n_{j+1/2}}
      + \Dt\norm{W'}_\infty \frac{\Phi_{1,1,\alpha}}{\Dz}\abs{\Delta
      w^n_{j+1/2}}\\
    &\le (\Delta\widehat{w})^n \biggl(1+\Dt\norm{W'}_\infty
    \Bigl(\, \sum_{j\ge 1}
      \frac{\Phi_{1,j,\alpha}-\Phi_{1,j+1,\alpha}}{\Dz}+\frac{\Phi_{1,1,\alpha}}{\Dz}    
      \, \Bigr)\biggr)\\
    &= (\Delta\widehat{w})^n \Bigl(1+2\Dt\norm{W'}_\infty
    \frac{\Phi_{1,1,\alpha}}{\Dz}    
      \Bigr),
  \end{align*}
  which implies \eqref{eq:xderiv}. We can also use this to prove
  \eqref{eq:tderiv},
  \begin{align*}
    \abs{w^{n+1}_i - w^n_i}
    &=\lambda\abs{\Delta\oWn_{i+1/2}}
      \le \lambda \sum_{j\ge i} \Phi_{ij\alpha} \abs{W(w^n_{j+1})-W(w^n_j)}\\
    &\le \lambda \norm{W'}_\infty(\Delta\widehat{w})^n \sum_{j\ge i} \Phi_{ij\alpha}
      \le \lambda \norm{W'}_\infty (\Delta\widehat{w})^0 
      \exp\Bigl(\frac{C}{\alpha} t^n\Bigr).
  \end{align*}
\end{proof}

The main theorem of this section states that the solutions to the
scheme \eqref{eq:wscheme} for the filtered variable converge to a
Lipschitz continuous weak solution of the filtered version of the
nonlocal Lagrangian PDE \eqref{eq:intro-new-nonlocal-PDE} (for a fixed
$\alpha$). To assist the convergence proof, define 
$w_{\Dt,\alpha}(z,t)$ to be the bi-linear interpolation 
of the points $\seq{(z_i,t^n,w^n_i)}$ 
with $j\in \Z$ and $n\ge 0$.

\begin{theorem}\label{thm:alphaconv}
  Let $0<T<\infty$ and assume that as $\Dt\to 0$, $\Dz\to 0$ in such 
  a way that the CFL condition \eqref{eq:CFLcond} is always 
  satisfied. Let $W(\cdot)$ be defined 
  by \eqref{eq:wi-average-Wdef} and consider 
  an initial function $1\le y_0\in BV(\mathbb{R})$. 
  Let $\alpha>0$ be fixed and assume furthermore that 
  the sequence of initial functions 
  $\seq{w_{\Dt,\alpha}(z,0)}_{\Dt>0}$ is such that
  $\abs{\pz w_{\Dt,\alpha}(z,0)}\le M$, where $M$ 
  does not depend on $\Dt$ (but on $\alpha$). 
  Suppose the averaging kernel $\Phi_\alpha$ 
  satisfies \eqref{eq:kernel-ass1}, \eqref{eq:kernel-ass2}. 
  Then there exists a Lipschitz continuous function 
  $w_\alpha:\R\times [0,T]\mapsto \R$ such that
  \begin{equation*}
    \lim_{\Dt\to 0} w_{\Dt,\alpha}=w_\alpha
    \quad \text{in $C(K\times[0,T])$, 
    \,\, $\forall K\subset \subset \R$}.
  \end{equation*}
  Moreover, $w_\alpha$ is a weak (distributional) solution of
  \begin{equation}\label{eq:walphapde}
      \begin{cases}
        \pt w_\alpha = \pz \overline{W(w_\alpha)},&\ \ z\in\R,\ \
        0<t\le
        T,\\
        w_\alpha(z,0)=\overline{y_0},&\ \ z\in\R,
      \end{cases}
  \end{equation}
  where the averaging (overline) operator is defined by
  \eqref{eq:average-op}, i.e.,
  \begin{equation*}
    \int_0^T\int_\R w_\alpha(z,t)\pt\test(z,t)
    -\overline{W(w_\alpha)}\pz\test(z,t)\,dzdt
    = \int_\R w_\alpha(z,T)\test(z,T)- w_\alpha(z,0)\test(z,0)\,dz
  \end{equation*}
  for all test functions $\test\in C^\infty_0(\R\times [0,T])$. 
  The solution is uniquely determined by the initial data.
\end{theorem}

\begin{proof}
  The uniform convergence $w_{\Dt,\alpha}\to w_\alpha$ follows by the
  Arzel\`a-Ascoli theorem and Lemma~\ref{lem:derivs}.

  For a fixed test function $\test$ define
  \begin{equation*}
    \test^n_i =
    \int_{z_{i-1/2}}^{z_{i+1/2}}\int_{t_n}^{t_{n+1}}
    \test(z,t) \,dtdz,
  \end{equation*}
  and write \eqref{eq:wscheme} as
  \begin{equation*}
    \frac{1}{\Dt}\left(w^{n+1}_i-w^n_i\right) -
    \frac{1}{\Dz}\left(\oWn_{i+1/2}-\oWn_{i-1/2}\right)=0. 
  \end{equation*}
  Multiply this with $\test^n_i$, sum over $n=0,1,\ldots,N-1$, where
  $N\Dt=T$, and over $i\in\Z$ and finally sum by parts to arrive at
  \begin{multline*}
    \sum_{i\in\Z}\sum_{n=1}^{N-1} w^n_i
    \frac{1}{\Dt}\left(\test^n_i-\test^{n-1}_i\right)
    -\sum_{i\in\Z}\sum_{n=0}^{N-1} \oWn_{i-1/2}
    \frac{1}{\Dz}\left(\test^n_{i}-\test^n_{i-1}\right)
    =\sum_{i\in\Z}\frac{1}{\Dt} w^N_i\test^{N-1}_i-\frac{1}{\Dt}
    w^0_i\test^{0}_i.
  \end{multline*}
  If we insert the definition of $\test^n_i$
  \begin{multline}
    \label{eq:discretesum}
    \sum_{i\in\Z}\sum_{n=1}^{N-1} w^n_i
    \int_{z_{i-1/2}}^{z_{i+1/2}}\int_{t_n}^{t_{n+1}}
    \frac{1}{\Dt}\left(\test(z,t)-\test(z,t-\Dt)\right) \,dtdz \\
    -\sum_{i\in\Z}\sum_{n=0}^{N-1} \oWn_{i-1/2}
    \int_{z_{i-1/2}}^{z_{i+1/2}}\int_{t_n}^{t_{n+1}}
    \frac{1}{\Dz}\left(\test(z,t)-\test(z-\Dz,t)\right) \,dtdz\\
    =\sum_{i\in\Z} w^N_i
    \int_{z_{i-1/2}}^{z_{i+1/2}}\frac{1}{\Dt}\int_{t_{N-1}}^{t_{N}}
    \test(z,t)\,dtdz - w^0_i
    \int_{z_{i-1/2}}^{z_{i+1/2}}\frac{1}{\Dt}\int_{0}^{\Dt}
    \test(z,t)\,dtdz.
  \end{multline}
  Now define the piecewise constant function (this is ``omega'', not
  ``double-u'')
  \begin{equation}\label{eq:piecewise-const} 
    \omega_{\Dt,\alpha}(z,t)=w^n_i\ \ \text{for $(z,t)\in
      [z_{i-1/2},z_{i+1/2})\times [t^n,t^{n+1})$.}
  \end{equation}
  Since $w_{\Dt,\alpha}$ is uniformly Lipschitz continuous with a
  Lipschitz constant $L$ not depending on $\Dt$ we have that
  $\abs{\omega_{\Dt,\alpha}(z,t)-w_{\Dt,\alpha}(z,t)}\le L\Dt$.
  Furthermore
  \begin{align*}
    \oWn_{i-1/2}
    & =\sum_{j\ge i} \Phi_{ij\alpha} W(w^n_j)
    =\sum_{j\ge i} \int_{z_{j-1/2}}^{z_{j+1/2}}
      \Phi_\alpha(\zeta-z_{i-1/2})\,d\zeta W(w^n_j)\\
    &=\int_{z_{i-1/2}}^\infty
      \Phi_\alpha(\zeta-z_{i-1/2})
      W(\omega_{\Dt,\alpha}(\zeta,t))\,d\zeta
    =\overline{W(\omega_{\Dt,\alpha})}(z_{i-1/2},t).
  \end{align*}
  Since $W$ is Lipschitz, it follows that $W(\omega_{\Dt,\alpha})$
  converges a.e.~and in $L^1_{\operatorname{loc}}$ 
  to $W(w_\alpha)$. Additionally, as the
  $\overline{\;\cdot\;}$ operator is continuous in $L^\infty$, we also
  have that $\overline{W(\omega_{\Dt,\alpha})}$ converges 
  a.e.~and in $L^1_{\operatorname{loc}}$ to
  $\overline{W(w_\alpha)}$. Hence also the piecewise constant
  function $\mathcal{W}$ defined by
  \begin{equation*}
    \mathcal{W}_\Dt(z,t)
    =\overline{W(\omega_{\Dt,\alpha})}(z_{i-1/2},t)
    \ \ \ \text{for $z\in
      [z_{i-1/2},z_{i+1/2})$,}
  \end{equation*}
  will converge in $L^1_{\operatorname{loc}}$ 
  to $\overline{W(w_\alpha)}$ as $\Dt\to 0$.
  With this notation, \eqref{eq:discretesum} can be rewritten
  \begin{multline}
    \label{eq:discreteint}
    \int_\R\int_{\Dt}^T \omega_{\Dt,\alpha}(z,t)
    \frac{1}{\Dt}\left(\test(z,t)-\test(z,t-\Dt)\right) \,dtdz \\
    - \int_\R\int_{0}^{T} \mathcal{W}_\Dt(z,t)
    \frac{1}{\Dz}\left(\test(z,t)-\test(z-\Dz,t)\right) \,dtdz\\
    =\int_\R \omega_{\Dt,\alpha}(z,T)
    \frac{1}{\Dt}\int_{t_N-\Dt}^{t_{N}} \test(z,t)\,dtdz - \int_\R
    \omega_{\Dt,\alpha}(z,0) \frac{1}{\Dt}\int_{0}^{\Dt}
    \test(z,t)\,dtdz.
  \end{multline}
  Now we can send $\Dt$ to $0$ in
  \eqref{eq:discreteint} and conclude that $w_\alpha$ is a (Lipschitz
  continuous) distributional solution of \eqref{eq:walphapde}.
  
  Finally, the assertion of uniqueness follows directly 
  from the $L^1$ contraction principle stated 
  in the upcoming Theorem \ref{thm:l1stability}.
\end{proof}

Finally, we will demonstrate a discrete entropy inequality for the
filtered scheme. Although this inequality will not be used directly in
our analysis, it serves as an important validation of the numerical
scheme (see also Corollary \ref{cor:bounds}). The inequality
shows that as the filter size becomes increasingly small, the
numerical scheme accurately captures the correct solution. This is a
crucial aspect, as it ensures the accuracy and well-balanced nature of
the scheme used.

\begin{lemma}
  \label{lem:discreteentropy}
  If the CFL-condition \eqref{eq:CFLcond} holds, then for any constant
  $c$
  \begin{equation*}
    \abs{w^{n+1}_i-c}
    \le \abs{w^n_i-c} + \lambda
    \sum_{j\ge i}\Phi_{ij\alpha}\left(Q_c(w^n_{j+1})-Q_c(w^n_i)\right),
  \end{equation*}
  where $Q_c(w)=\sgn{w-c}(W(w)-W(c))$.
\end{lemma}
\begin{proof}
  For $\bm{w}=\seq{w_i}_{i\in\Z}$ we define
  \begin{equation*}
    G(\bm{w})_i = w_i+\lambda 
    \sum_{j\ge i}\Phi_{ij\alpha} \left(W(w_{j+1})-W(w_j)\right),
  \end{equation*}
  and observe that the mapping $\bm{w}\mapsto G(\bm{w})$ is monotone
  in the sense that if $v_i\le w_i$ for all $i$, then
  $G(\bm{v})_i\le G(\bm{w})_i$ for all $i$. Using $G$ the scheme reads
  $w^{n+1}_i = G\left(\bm{w}^n\right)_i$. Let $\bm{c}$ denote the 
  constant vector with all entries equal to
  the number $c$, $\max\seq{\bm{a},\bm{b}}_i=\max\seq{a_i,b_i}$, 
  and $\min\seq{\bm{a},\bm{b}}_i=\min\seq{a_i,b_i}$. Then we have
  \begin{align*}
    G\left(\max\seq{\bm{w}^n,\bm{c}}\right)_i
    &=\max\seq{\bm{w}^n_i,c} + \lambda \sum_{j\ge i}\Phi_{ij\alpha}
      \left(W\left(\max\seq{w^n_{j+1},c}\right)
      -W\left(\max\seq{w^n_{j},c}\right)\right)\\
    &\le \max\seq{G(\bm{w}^n)_i,c},\\
    G\left(\min\seq{\bm{w}^n,\bm{c}}\right)_i
    &=\min\seq{w^n_i,c} + \lambda \sum_{j\ge i}\Phi_{ij\alpha}
      \left(W\left(\min\seq{w^n_{j+1},c}\right)
      -W\left(\min\seq{w^n_{j},c}\right)\right)\\
    &\ge \min\seq{G(\bm{w}^n)_i,c}.
  \end{align*}
  Subtracting these inequalities we get
  \begin{align*}
    \abs{w^{n+1}_i-c}
    &=\max\seq{G(\bm{w}^n)_i,c}-\min\seq{G(\bm{w}^n)_i,c}\\
    &\le \max\seq{w^n_i,c}-\min\seq{w^n_i,c} \\
    &\qquad + \lambda \sum_{j\ge i} \Phi_{ij\alpha}
      \bigl[\left(W\left(\max\seq{w^n_{j+1},c}\right)
      -W\left(\min\seq{w^n_{j+1},c}\right)\right)
    \\
    &\qquad\hphantom{+ \lambda \sum_{j\ge i} \Phi_{ij\alpha}\Bigl[}
      -\left(W\left(\max\seq{w^n_{j},c}\right)
      -W\left(\min\seq{w^n_{j},c}\right)\right)\bigr]\\
    &=\abs{w^n_i-c} + \lambda \sum_{j\ge i}
      \Phi_{ij\alpha}\left(Q_c\left(w^n_{j+1}\right)
      -Q_c\left(w^n_i\right)\right).
  \end{align*}
\end{proof}

Recall that $w_\alpha$ is the Lipschitz continuous weak solution of
\eqref{eq:walphapde}, which is the filtered version of the nonlocal
Lagrangian PDE model \eqref{eq:intro-new-nonlocal-PDE}.  Using similar
reasoning as in the proof of Theorem~\ref{thm:alphaconv}, it can be
demonstrated that $w_\alpha$ satisfies the Kru{\v{z}}kov entropy
inequalities $\pt\abs{w_\alpha-c}\le \pz\overline{Q_c(w_\alpha)}$, 
for $c\in \R$. In Section~\ref{sec:zero-filter} we will show 
that a refined version of this entropy inequality 
is satisfied by any Lipschitz continuous weak
solution of \eqref{eq:walphapde}.

\begin{remark}
  The unique form of the ``filtered equation'', i.e., the nonlocal PDE
  \eqref{eq:walphapde}, suggests it can be interpreted as a fractional
  conservation law, where the spatial derivative is a fractional
  derivative operator.  Recent studies, such as those referenced in
  \cite{Alibaud:2007mi,Alibaud:2012,Karlsen:2009vf} and many other
  others, have explored perturbations of conservation laws through the
  use of fractional diffusion or more general L{\'e}vy operators.
  This connection will be further clarified in the following.

  Recall that the transport part of the nonlocal PDE
  \eqref{eq:walphapde} can be written in the form
$$
\pz \overline{W(w_\alpha)}(z,t) = \int_0^\infty
\bigl(-\Phi'_\alpha\bigr) (\zeta) \bigl[W(w_\alpha(z+\zeta,t))
-W(w_\alpha(z,t))\bigr]\,d\zeta.
$$
For motivational reasons, let us specify the kernel as
$\Phi_\alpha(z)=e^{-z/\alpha}/\alpha$.  Then it follows that
$\bigl(-\Phi_\alpha'\bigr)(z)= \Phi_\alpha(z)/\alpha$ and
$\int_0^\infty \bigl(-\Phi_\alpha'\bigr)(z)\, dz=1/\alpha$, but note
that $\int_0^\infty z\bigl(-\Phi_\alpha'\bigr)(z)\, dz=1$. 

Introducing the measure $\pi(dz)$ on $\R$ defined by
$$
\pi(dz)= \bigl(-\Phi_\alpha'\bigr)(z) \chi_{(-\infty,0]}(z)\,dz,
$$
which satisfies first moment condition
$\int_{\R} \abs{z} \, \pi(dz)<\infty$, we may express
the term $\pz \overline{W(w_\alpha)}(z,t)$ as 
$\int_{\R}\bigl[W(w_\alpha(z+\zeta,t))
-W(w_\alpha(z,t))\bigr]\,\pi(d\zeta)$. Dropping the
$\alpha$-subscript, the nonlocal PDE \eqref{eq:walphapde} now becomes
$$
\pt w=\int_{\R}\bigl[W(w(z+\zeta,t)) -W(w(z,t))\bigr]\,\pi(d\zeta).
$$
The measure $\pi(dz)$ depends discontinuously on the position $z$,
which contrasts with studies such as
\cite{Alibaud:2007mi,Alibaud:2012,Karlsen:2009vf}.  Aiming for a
generalised traffic flow model, we may treat $\pi(d\zeta)$ as a
general L{\'e}vy measure, which describes the distribution of jumps in
a L{\'e}vy process. In particular, one-sided L{\'e}vy processes
(subordinators) may be relevant.  A L{\'e}vy process is a stochastic
process with independent and stationary increments and can be thought
of as an extension of Brownian motion. Lévy processes and fractional
derivatives can be used to model various types of anomalous diffusion
phenomena, including the spread of information in complex
transportation systems impacted by factors such as network structure,
individual behavior, and external disruptions. Fractional derivatives
are non-local operators that account for long-range interactions and
memory effects. A famous example of a L{\'e}vy measure is provided by
$\pi(dz)=|z|^{-(1+\gamma)}\, \chi_{|z|<1} \, dz$, for
$\gamma \in (0,2)$. This example is related to the fractional
Laplacian $\Delta_\alpha:=-(-\Delta)^{\frac{\gamma}{2}}$on $\R$.  For
more information on Lévy processes, including one-sided processes
(subordinators), see \cite{Sato:1999qd}.
\end{remark}

\section{The nonlocal Lagrangian PDE for $y=1/u$}\label{sec:PDE-for-y}

Let us discuss the relationship between the scheme for the
filtered variable $w=\overline{y}$ and a (fully discrete) scheme for
the original variable $y=1/u$.  
Assuming that the nonlocal operator $\overline{\;\cdot\;}$ 
is invertible (which is true for certain averaging kernels, such 
as $\Phi_\alpha(z)=e^{-z/\alpha}/\alpha$), then we 
can directly recover the values $\left\{y_i^n\right\}$ from the values
$\left\{w_i^n\right\}$ computed via the scheme
\eqref{eq:wscheme}. Alternatively, we can start from a fully
discrete version of \eqref{eq:ode1} for $y_i^n=1/u_i^n$:
\begin{equation}\label{eq:y-scheme}
  y^{n+1}_i = y^n_i+\lambda
  \left(W(w_{i+1}^n)-W(w_i^n)\right),
  \quad  i\in \Z,\,\, n\in \N,
\end{equation}
where, for $n=0$, $\left\{y^0_i\right\}$ is an approximation of the
initial function $y_0=1/u_0$, and $w_i^n=
\sum_{j\ge i} \Phi_{ij\alpha}y_j^n$, 
$\Phi_{ij\alpha}=\int_{z_{j-1/2}}^{z_{j+1/2}} 
\Phi_\alpha(\zeta-z_{i-1/2})\,d\zeta$, $i\in \Z$.
This is an explicit upwind (Godunov-type) scheme for approximating
solutions $y=1/u$ to the nonlocal Lagrangian PDE
\eqref{eq:intro-new-nonlocal-PDE}.  Applying the averaging operator
$\overline{\;\cdot\;}$ to \eqref{eq:y-scheme} leads to the scheme
\eqref{eq:wscheme} for the filtered variable
$w_i^n=\overline{y}_i^n=\overline{\frac{1}{u_i^n}}$.

The ($\alpha$-independent) bound of the 
subsequent lemma implies that the scheme \eqref{eq:y-scheme} 
converges weakly to a limit $y_\alpha$, which will be 
proven later to be a solution of the 
nonlocal PDE \eqref{eq:intro-new-model}.

\begin{lemma}\label{lm:1} 
 Let $1\le y_0\in BV(\R)$ be given.  If the CFL-condition
 \eqref{eq:CFLcond} holds, then
 \begin{equation}\label{eq:4}
   \inf_{z\in\R} y_0(z)\le y^n_i\le \sup_{z\in\R} y_0(z),
 \end{equation}
 for every $\alpha>0$ and $i\in\Z$, $n\ge 0$, where
 $\left\{y^n_i\right\}_{i,n}$ solves \eqref{eq:y-scheme}.
\end{lemma}
\begin{proof}
 Introduce the notation
 \begin{equation*}
   I_j=\int_{z_j}^{z_{j+1}} \Phi_\alpha(\zeta)\,d\zeta\ \ \
   \text{and}
   \ \ \ A^n_i = \frac{W(w^n_{i+1})-W(w^n_{i})}{w^n_{i+1}-w^n_i}\ge 0.
 \end{equation*}
 By a summation by parts, the scheme for $y^n_i$ \eqref{eq:y-scheme}
 can be written
 \begin{align*}
   y^{n+1}_i-y^n_i
   &=\lambda \left(W\left(w^n_{i+1}\right)-W\left(w^n_i\right)\right)
     =\lambda A^n_i \left(w^n_{i+1}-w^n_i\right)\\
   &=\lambda A^n_i \sum_{j=1}^\infty I_{j-1}\left(y^n_{i+j} -
     y^n_{i+j-1}\right)\\
   &=\lambda A^n_i\Bigl( \sum_{j=1}^\infty \left(I_{j-1}-I_j\right)
     y^n_{i+j} - I_0y^n_i\Bigr)\hspace{3cm}
     \Bigl( I_0=\sum_{j=1}^\infty\left(I_{j-1}-I_j\right)\Bigr) \\
   &=\lambda A^n_i \sum_{j=1}^\infty \left(I_{j-1}-I_j\right)
     \left( y^n_{i+j}-y^n_i\right),
 \end{align*}
 or
 \begin{equation*}
   y^{n+1}_i=G(A^n_i,y_i^n,y_{i+1}^n,y_{i+2}^n,\ldots),
 \end{equation*}
 with the bilinear function $G$ defined by
 \begin{equation*}
   G(A,\bm{y})=\left(1-\lambda A I_0\right)y_1 + \lambda A \sum_{j=1}^\infty
   \left(I_{j-1}-I_j\right)y_{j+1}=y_1+ \lambda A \sum_{j=1}^\infty
   \left(I_{j-1}-I_j\right)\left(y_{j+1}-y_1\right),
 \end{equation*}
 for a number $A$ and a vector $\bm{y}=\seq{y_i}_{i=1}^\infty$.
 Observe that $G(A,y,y,y,\ldots)=y$ and that \emph{for fixed $A\ge 0$},
 the map $\seq{y_i}\mapsto G(A,\seq{y_i})$ (by the CFL-condition and
 the fact that $I_{j-1}\ge I_j$) is monotone increasing in each
 argument $y_1,y_2,y_3,\ldots$.  Set
 \begin{equation*}
   \check{y}=\inf_{i\in\Z} y^n_i\ \ \ \text{and}\ \ \ \hat{y}=\sup_{i\in\Z} y^n_i.
 \end{equation*}
 For any $i\in\Z$ and any $n\ge 0$
 \begin{align*}
   \check{y}
   &=G\left(A^n_i,\check{y},\check{y},\check{y},\ldots\right)\le
     G\left(A^n_i,y^n_i,y^n_{i+1},y^n_{i+2},\ldots\right)\\
   &=y^{n+1}_i= G\left(A^n_i,y^n_i,y^n_{i+1},y^n_{i+2},\ldots\right)\le
     G\left(A^n_i,\hat{y},\hat{y},\hat{y},\ldots\right)=\hat{y}. 
 \end{align*}
Hence
$\inf_{i\in\Z} y^n_i\le \inf_{i\in\Z} y^{n+1}_i\le \sup_{i\in\Z}
y^{n+1}_i \le \sup_{i\in\Z} y^{n}_i$, and the lemma follows by induction.
\end{proof}

We denote by $w_{\Dt,\alpha}(z,t)$ the bi-linear interpolation of 
the points $\seq{(z_i,t^n,w^n_i)}$ with $j\in \Z$, $n\ge 0$, 
and $t^n=n\Dt$, recalling \eqref{eq:y-scheme}. 
Based on Theorem \ref{thm:alphaconv}, we conclude that $w_{\Dt,\alpha}(z,t)$ 
converges uniformly on compacts to a Lipschitz continuous 
limit $w_\alpha(z,t)$ as $\Dt\to 0$. The piecewise constant 
interpolation of the points $\seq{(z_i,t^n,w^n_i)}$ is denoted 
by $\omega_{\Dt,\alpha}(z,t)$ and it converges a.e.~and thus in 
$L^1(K\times [0,T])$, $\forall K\subset\subset \R$. 
The piecewise constant interpolation of the points $\seq{(z_i,t^n,y^n_i)}$ 
is denoted by $y_{\Dt,\alpha}(z,t)$. Due to the 
estimate \eqref{eq:4}, $y_{\Dt,\alpha}$ is bounded in 
$L^\infty(\R\times \R_+)$ uniformly in $\Dt$ (and $\alpha$). 
Hence, there exists a subsequence $\left\{y_{\Dt_m,\alpha}\right\}_{m\in \N}$ 
that converges weak-$\star$ in $L^\infty(\R\times \R_+)$ to some 
limit $y_\alpha$. This implies that the functions 
$y_\alpha$, $w_\alpha$ satisfy (weakly) the nonlocal Lagrangian 
PDE \eqref{eq:intro-new-model} with $w_\alpha=\overline{y_\alpha}$. 
By the uniqueness of solutions (from Remark \ref{rm:stab}), the 
entire sequence $\left\{y_{\Dt,\alpha}\right\}$ converges. 
In summary, we have proved the following proposition:

\begin{proposition}\label{prop:weak-conv-yalpha}
Suppose the assumptions of Theorem \ref{thm:alphaconv} hold.
There exists a pair $\bigl(y_\alpha,w_\alpha\bigr)$, with
$1\le y_\alpha\in L^\infty(\R\times \R_+)$ and $w_\alpha\in 
\bigl(\operatorname{Lip}_{\operatorname{loc}}
\cap L^\infty\bigr)(\R\times \R_+)$, 
such that the following convergences hold as $\Dt\to 0$ 
(with $\alpha>0$ fixed):
\begin{align*}
	&y_{\Dt,\alpha}\to y_\alpha\quad 
	\text{weak-$\star$ in $L^\infty(\R\times\R_+)$},
	\\ &
	\text{$w_{\Dt,\alpha}\to w_\alpha$ uniformly 
	on compacts of $\R\times\R_+$}.
\end{align*}
Besides, $\bigl(y_\alpha,w_\alpha\bigr)$ is a weak solution of 
\begin{equation}\label{eq:3c}
	\begin{cases}
		\pt y_{\alpha} = \pz W(w_{\alpha}),&\quad z\in \R,\,t>0,
		\\[3pt]
		\displaystyle 
		w_{\alpha}(z,t)=\int_{z}^{\infty}
		\Phi_\alpha\left({\zeta-z}\right)y_{\alpha}(\zeta,t)
		\,d\zeta, & \quad z\in \R,\,t>0,
		\\[3pt]
		y_{\alpha}(z,0)=y_0(z), & \quad z\in \R.
  \end{cases}
\end{equation}
Weak solutions from the class 
$L^\infty(\R\times \R_+)\times
\bigl(\operatorname{Lip}_{\operatorname{loc}}
\cap L^\infty\bigr)(\R\times \R_+)$ 
are uniquely determined by their initial data.
\end{proposition}

\begin{remark}\label{rm:stab}
To conclude this section, we examine the stability 
of the nonlocal Lagrangian PDE \eqref{eq:3c} in response 
to perturbations in the averaging kernel $\Phi$. 
Suppose $\Phi_1$ and $\Phi_2$ both adhere to 
the same assumptions outlined in \eqref{eq:kernel-ass1} as $\Phi$. 
Consider the solutions $y_{1,\alpha}$ and $y_{2,\alpha}$ 
to \eqref{eq:3c} with $\Phi_{1,\alpha}$ 
and $\Phi_{2,\alpha}$ as the averaging kernels, 
see \eqref{eq:kernel-ass2}, and $y_{1,0}$, $y_{1,0}$ as the initial data. 
A simple calculation yields the stability estimate
\begin{align*}
	\norm{y_{1,\alpha}(\cdot,t)-y_{2,\alpha}(\cdot,t)}_{L^1(\R)}
	\le & e^{ct/\alpha} \norm{y_{1,0}-y_{2,0}}_{L^1(\R)}
	\\ & 
	+c\alpha \bigl(e^{ct/\alpha}-1\bigr)
	\norm{\Phi_1-\Phi_2}_{L^1(\R)}
	+c\bigr(e^{ct/\alpha}-1\bigr)
	\norm{\Phi_1'-\Phi_2'}_{L^1(\R)}.
\end{align*}
where $c$ does not depend on $\alpha$.
\end{remark}

\section{Eulerian formulation}\label{sec:Eulerian}

One can transform the nonlocal Lagrangian  
PDE \eqref{eq:3c}---or \eqref{eq:intro-new-nonlocal-PDE}---into 
an Eulerian PDE via a change of variable, assuming that smooth 
solutions exist. However, this results in a complex and 
difficult-to-analyse Eulerian PDE. We only display this PDE here 
to highlight differences from other nonlocal Eulerian traffic flow equations, 
like \eqref{eq:nonlocalCL-intro}. Wagner's result \cite{Wagner:1987aa} 
provides a rigourous framework for converting Lagrangian PDEs 
to Eulerian PDEs for weak solutions. 

The Eulerian form of \eqref{eq:intro-new-nonlocal-PDE} reads
\begin{align}
	\label{eq:6ee}
	&\pt{\widetilde u}
	+\px\left({\widetilde u}\, 
	V\left(  \,\left[\,\overline{\frac{1}{{\widetilde u}(x,t)}}
	\,\right]^{-1}\,\right)\right)=0,
	\\ 
	\label{eq:6eee}
	&\overline{\frac{1}{{\widetilde u}(x,t)}}
	=\int_{x}^{\infty}\Phi_\alpha
	\left(\, \int_x^\sigma \widetilde{u}(\theta,t)
	\,d\theta\,\right)\,d\sigma.
\end{align}
We may rewrite \eqref{eq:6eee} in a slightly clearer form.
Since $0<u_*\le{\widetilde u}\le1$,
the function $\sigma\mapsto\int_x^\sigma{{\widetilde u}(\theta,t)}
\,d\theta$ is invertible and $\int_x^\infty \widetilde{u}(\theta,t)
\,d\theta=\infty$. Therefore, we may express 
$\overline{\frac{1}{\widetilde{u}}}$ at the point $(x,t)$ as a weighted 
harmonic mean of $\widetilde{u}$ around different points 
$\ell\mapsto \bigl(\sigma(\ell, x,t),t\bigr)$:
\begin{equation}\label{eq:nonlocal-new-ver}
	\overline{\frac{1}{{\widetilde u}(x,t)}}
	=\int_{0}^{\infty}\Phi_\alpha\left(\ell\right)
	\frac{1}{\widetilde{u}(\sigma(\ell, x,t),t)}\,d\ell,
\end{equation}
where $\sigma(\ell,x,t)$ satisfies 
$\ell=\int_x^{\sigma(\ell, x,t)}\widetilde{u}(\theta,t)
\,d\theta$; the new variable $\ell$ should not be confused with 
the $\ell$ appearing in \eqref{eq:uidef}.

\begin{remark}
Formally, by sending $\alpha\to0$ 
in \eqref{eq:6ee} and \eqref{eq:6eee}, we arrive 
at the local LWR equation \eqref{eq:LWR-eqn}.
To see this, note that the relation 
$\ell=\int_x^\sigma \widetilde{u}(\theta,t)\,d\theta$ 
implies $0=\int_x^{\sigma(0, x,t)}{{\widetilde u}(\theta,t)}\,d\theta$, 
from which we conclude that $\sigma(0,x,t)=x$. 
As a result, sending $\alpha\to0$ in 
\eqref{eq:nonlocal-new-ver} yields
\begin{equation*}
	\int_{0}^{\infty}\Phi_\alpha\left(\ell\right)
	\frac{1}{{\widetilde u}(\sigma(\ell, x,t),t)}
	\,d\ell
	\longrightarrow 
	\frac{1}{\widetilde u(\sigma(0,x,t),t)}
	=\frac{1}{\widetilde u(x,t)},
\end{equation*}
and then \eqref{eq:6ee} becomes \eqref{eq:LWR-eqn}: 
$\pt \widetilde{u}+\px\bigl(\widetilde{u} V(\widetilde{u})\bigr)=0$.
\end{remark}

Under the assumption of smooth solutions, we will outline 
a derivation of \eqref{eq:6ee} and \eqref{eq:6eee}. 
For a derivation that works for weak solutions, see 
\cite{Wagner:1987aa}. Let $\psi_t(z)$ satisfy
\begin{equation}\label{eq:1e}
	\pz \psi_t(z)=\frac{1}{u(z,t)},\qquad 
	\pt \psi_t(z)=V\left( \,\left[\,\overline{\frac{1}{u(z,t)}}
	\,\right]^{-1}\, \right).
\end{equation}
Denote by $\psi_t^{-1}(\cdot)$ the inverse of $\psi_t(\cdot)$, so that
\begin{equation}\label{eq:2e}
	\psi_t(\psi_t^{-1}(x))=x.
\end{equation}
Define
\begin{equation}\label{eq:3e}
	\widetilde{u}(x,t)=u(\psi_t^{-1}(x),t).
\end{equation}
Differentiating \eqref{eq:2e} with respect to $x$ yields
$\pz \psi_t(\psi_t^{-1}(x))\px \psi_t^{-1}(x)=1$.
Thus, by \eqref{eq:1e}, $\px \psi_t^{-1}(x)$ equals 
$1/\pz \psi_t(\psi_t^{-1}(x))=u(\psi_t^{-1}(x),t)$, and, thanks 
to \eqref{eq:3e},
\begin{equation}\label{eq:4e}
	\px \psi_t^{-1}(x)=\widetilde{u}(x,t).
\end{equation}
Differentiating \eqref{eq:2e} with respect to $t$ yields
$\pz \psi_t(\psi_t^{-1}(x))\pt \psi_t^{-1}(x)
+\pt \psi_t(\psi_t^{-1}(x))=0$. 
Hence, using \eqref{eq:1e} and \eqref{eq:3e},
\begin{equation}\label{eq:5e}
	\pt \psi_t^{-1}(x)
	=-u(\psi_t^{-1}(x),t)
	V\left( \,\left[\,\overline{\frac{1}{u(\psi_t^{-1}(x),t)}}
	\,\right]^{-1}\, \right)
	=-\widetilde{u}(x,t)
	V\left( \,\left[\,\overline{\frac{1}{{\widetilde u}(x,t)}}\,
	\right]^{-1}\,\right).
\end{equation}
Using \eqref{eq:3e}, \eqref{eq:5e}, 
and \eqref{eq:intro-new-nonlocal-PDE} 
to express $\partial_t u(z,t)$ as $-u^2(z,t)\partial_z V\Bigl( \,\left[\,
\overline{\frac{1}{u(z,t)}}\,\right]^{-1}\, \Bigr)$, we obtain
\begin{align*}
	\pt\widetilde{u}(x,t) 
	& =\pz u(\psi_t^{-1}(x),t)\pt \psi_t^{-1}(x)
	+\pt u(\psi_t^{-1}(x),t)
	\\ &
	 = -\pz u(\psi_t^{-1}(x),t) u(\psi_t^{-1}(x),t)
	 \pz V\left(\,\left[\,\overline{\frac{1}{u(\psi_t^{-1}(x),t)}}
	 \,\right]^{-1}\, \right)
	 \\ & \qquad 
	 -u^2(\psi_t^{-1}(x),t)
	 V\left( \,\left[\,\overline{\frac{1}{u(\psi_t^{-1}(x),t)}}
	 \,\right]^{-1}\, \right)
	 \\ & = 
	 -u(\psi_t^{-1}(x),t)\pz\left(u(\psi_t^{-1}(x),t)
	 V\left( \,\left[\,\overline{\frac{1}{u(\psi_t^{-1}(x),t)}}
	 \,\right]^{-1}\,\right)\right).
\end{align*}
In view of \eqref{eq:4e} and \eqref{eq:3e}, this yields
\begin{align*}
	\pt\widetilde{u}(x,t) 
	& = -\px \psi_t^{-1}(x)\pz\left(u(\psi_t^{-1}(x),t)
	V\left( \,\left[\,\overline{\frac{1}{u(\psi_t^{-1}(x),t)}}
	\,\right]^{-1}\, \right)\right)
	\\ & 
	=-\px\left(u(\psi_t^{-1}(x),t)
	V\left[\,\overline{\frac{1}{u(\psi_t^{-1}(x),t)}}
	\,\right]^{-1}\,\right)
	=-\px\left(\widetilde{u}(x,t)
	V\left( \,\left[\,\overline{\frac{1}{\widetilde{u}(x,t)}}
	\,\right]^{-1}\,\right)\right), 
\end{align*}
which is \eqref{eq:6ee}. Furthermore, using 
\eqref{eq:3e} and \eqref{eq:intro-new-nonlocal-PDE-II},
\begin{align*}
	\overline{\frac{1}{\widetilde{u}(x,t)}}
	& =\overline{\frac{1}{u(\psi_t^{-1}(x),t)}}
	= \int_{\psi_t^{-1}(x)}^{\infty}
	\Phi_\alpha\left({\zeta-\psi_t^{-1}(x)}\right)
	\frac{1}{u(\zeta,t)}\,d\zeta.
\end{align*}
Introduce the change of variable $\zeta=\psi_t^{-1}(\sigma)$ 
for $\sigma\in [x,\infty)$, so that $d\zeta=\px \psi_t^{-1}(\sigma) 
\, d\sigma=\widetilde{u}(\sigma,t)\, d\sigma$, 
cf.~\eqref{eq:4e} and \eqref{eq:3e}. Then
\begin{align*}
	\overline{\frac{1}{\widetilde{u}(x,t)}}
	& = \int_{x}^{\infty}
	\Phi_\alpha\left({\psi_t^{-1}(\sigma)
	-\psi_t^{-1}(x)}\right)\,d\sigma 
	=\int_{x}^{\infty}\Phi_\alpha\left(
	\int_x^\sigma{\px \psi_t^{-1}(\theta)}\,d\theta
	\right)\,d\sigma
	\\ & =\int_{x}^{\infty}\Phi_\alpha\left(
	\int_x^\sigma{\widetilde{u}(\theta,t)}
	\,d\theta\right)\,d\sigma,
\end{align*}
which is \eqref{eq:6eee}.

\begin{remark}
For comparative purposes, let us discuss the relationship 
between Lagrangian and Eulerian variables in 
the ``standard" nonlocal traffic flow equations 
\eqref{eq:nonlocalCL-intro}, starting with the first equation. 
The macroscopic Lagrangian model corresponding to the nonlocal 
FtL model \eqref{eq:nonlocal-FtL-density} is
\begin{equation*}
  \partial_t \left(\frac{1}{u(z,t)}\right)
  -\partial_z V\left( \overline{u}(z,t) \right)=0, 
  \quad z\in \R, \,\, t>0, 
\end{equation*}  
where
\begin{equation}\label{eq:6eee-standard-I}
	\overline{u}(z,t)=\int_{\psi_t(z)}^\infty 
	\Phi_{\alpha}(\zeta-\psi_t(z))u(\psi_t^{-1}(\zeta),t)\, d\zeta,
\end{equation}
and $\psi_t(z)$ satisfies the equations
\begin{equation*}
	\pz \psi_t(z)=\frac{1}{u(z,t)},
	\qquad \pt \psi_t(z)=V\left( \overline{u}(z,t) \right).
\end{equation*}
By repeating the steps that led to 
\eqref{eq:6ee} and \eqref{eq:6eee}, with necessary 
adjustments to account for the differences 
between \eqref{eq:intro-new-nonlocal-PDE-II} 
and \eqref{eq:6eee-standard-I}, we derive the 
first Eulerian PDE in \eqref{eq:nonlocalCL-intro} for 
the function $\widetilde{u}(x,t)=u(\psi_t^{-1}(x),t)$. 
These adjustments include expressing 
\eqref{eq:6eee-standard-I} as
\begin{align*}
	\overline{\widetilde{u}}(x,t)
	=\overline{u(\psi_t^{-1}(x),t)}
	& =\int_{\psi_t(z)}^\infty 
	\Phi_{\alpha}(\zeta-\psi_t(z))
	u(\psi_t^{-1}(\zeta),t)\, d\zeta
	=\int_x^\infty
	\Phi_{\alpha}(\zeta-x)
	\widetilde{u}(\zeta,t)\, d\zeta.
\end{align*}

Similarly, the macroscopic Lagrangian model correponding to 
\eqref{eq:nonlocal-FtL-velocity} takes the form
\begin{equation*}
	\partial_t \left(\frac{1}{u(z,t)}\right)
	-\partial_z \overline{V\left(u(z,t)\right)}=0, 
	\quad z\in \R, \,\, t>0, 
\end{equation*}  
where
$$
\overline{V\left( u(z,t) \right)}
=\int_{\psi_t(z)}^\infty 
\Phi_{\alpha}(\zeta-\psi_t(z))
V(u(\psi_t^{-1}(\zeta),t))\, d\zeta,
$$
and $\psi_t(z)$ satisfies
\begin{equation*}
	\pz \psi_t(z)=\frac{1}{u(z,t)},
	\qquad 
	\pt \psi_t(z)=\overline{V\left( u(z,t) \right)}.
\end{equation*}
Using the same reasoning, the second Eulerian 
PDE in \eqref{eq:nonlocalCL-intro} is derived.
\end{remark}

\section{Zero-filter limit of the nonlocal model}\label{sec:zero-filter}
In this section, we will examine a sequence of Lipschitz continuous
weak solutions $w_\alpha$, indexed by the filter size $\alpha>0$, of
the filtered version of the nonlocal Lagrangian PDE
\eqref{eq:intro-new-nonlocal-PDE}, see \eqref{eq:walphapde} 
and Theorem \ref{thm:alphaconv}. We will
prove that these solutions have $\alpha$-independent estimates,
precise entropy equalities, and converge to the unique entropy
solution of the original LWR equation \eqref{eq:LWR-eqn} in Lagrangian
coordinates.

Let $(\eta,Q)$ be an entropy/entropy-flux pair, i.e., $\eta$ is a
convex, twice continuously differentiable function and $Q$ is a
function satisfying $Q'(w)=\eta'(w)W'(w)$.  Multiply
\eqref{eq:walphapde} with $\eta'(w(z,t))$ to get
\begin{align*}
  \pt \eta(w_\alpha)
  &=\pz \overline{Q(w_\alpha)}
    +\eta'(w_\alpha)\pz \overline{W(w_\alpha)} 
    -\pz \overline{Q(w_\alpha)}\\
  &=\pz \overline{Q(w_\alpha)}\\
  &\quad +
    \int_0^\infty \Phi'_\alpha(\zeta)
    \bigl[ \left(\eta'(w_\alpha(z,t)) W(w_\alpha(z,t))
    -Q(w_\alpha(z,t))\right)\\
  &\quad\hphantom{\int_0^\infty \Phi'_\alpha(\zeta)}\;
    - \left(\eta'(w_\alpha(z,t))W(w_\alpha(z+\zeta,t))
    -Q(w_\alpha(z+\zeta,t))\right)\bigr]\,d\zeta\\
  &=\pz \overline{Q(w_\alpha)} +  \int_0^\infty
  \Phi'_\alpha(\zeta)
  H(w_\alpha(z,t),w_\alpha(z+\zeta,t))\,d\zeta,
\end{align*}
where, recalling that $W'(\cdot)\ge 0$,
\begin{align*}
  H(a,b)
  &=\left[ \bigl(\eta'(a) W(a)-Q(a)\bigr)
    - \bigl(\eta'(a)W(b)-Q(b)\bigr)\right]\\
  &= \int^a \bigl(\eta'(a)-\eta'(\sigma)\bigr)W'(\sigma)\,d\sigma 
  - \int^b \bigl(\eta'(a)-\eta'(\sigma)\bigr)W'(\sigma)\,d\sigma\\
  &=\int_a^b \bigl(\eta'(\sigma)-\eta'(a)\bigr)W'(\sigma)\,d\sigma
    =\int_a^b \int_a^\sigma \eta''(\mu)\,d\mu 
    \, W'(\sigma)\,d\sigma\ge 0.
\end{align*}
Since $\Phi'_a\le 0$, we have proved that a solution $w_\alpha$ of
\eqref{eq:walphapde} satisfies an entropy (in)equality.

\begin{theorem}\label{thm:entropyapriori}
  Let $w_\alpha$ be a Lipschitz continuous distributional solution of
  \eqref{eq:walphapde}, see Theorem \ref{thm:alphaconv}. 
  Then for any entropy/entropy-flux pair $(\eta,Q)$
  \begin{equation}\label{eq:entropy-equality}
    \pt \eta(w_\alpha(z,t))+D(z,t)
    =\pz\overline{Q(w_\alpha)}(z,t),
  \end{equation}
  where
  \begin{equation*}
    D(z,t)=
    \int_0^\infty \bigl(-\Phi_\alpha'\bigr)(\zeta)
    \int_{w_\alpha(z,t)}^{w_\alpha(z+\zeta,t)}\int_{w_\alpha(z,t)}^{\sigma}
    \eta''(\mu) W'(\sigma)\,d\mu\,d\sigma \,d\zeta \ge 0.
  \end{equation*}
\end{theorem}

\begin{remark}
  For concrete choices of the entropy $\eta$ we obtain more precise
  estimates. If we suppose
  $\inf_{\mu,\sigma} [\eta''(\mu)W'(\sigma)] \ge 2c > 0$ for some
  constant $c$, then
  \begin{equation*}
    \int_a^b \int_a^\sigma \eta''(\mu)\,d\mu W'(\sigma)\,d\sigma\ge c(b-a)^2,
  \end{equation*}
  and consequently
  \begin{equation*}
    D(z,t)\ge c\int_0^\infty \bigl(-\Phi'_\alpha\bigr)(\zeta)
    \bigl(w_\alpha(z+\zeta,t)-w_\alpha(z,t)\bigr)^2\,d\zeta
    \,\, (\ge0).
  \end{equation*}
  For example, specifying $\eta(w)=w^2/2$ and integrating
  \eqref{eq:entropy-equality} over $[-R,R]\times [0,T]$, 
  we obtain the additional a priori estimate
  \begin{equation*}
    \int_0^T\int_{-R}^R
    \int_0^\infty \bigl(-\Phi'_\alpha\bigr)(\zeta)
    \bigl(w_\alpha(z+\zeta,t)-w_\alpha(z,t)\bigr)^2\,d\zeta
    \, dz\,dt\le C_R. 
  \end{equation*}
   
  If we use the Kru\v{z}kov entropy
  \begin{equation*}
    \eta(w)=\abs{w-k}, \ \ \eta'(w)=\sgn{w-k}, 
    \ \ \eta''(w)=2\delta_k(w),
  \end{equation*}
  we obtain
  \begin{equation*}
    H(w_\alpha(z,t),w_\alpha(z+\zeta,t))=
    \begin{cases}
      2\abs{w_\alpha(z+\zeta,t)-w_\alpha(z,t)} & \text{if $k$ is
        between $w_\alpha(z,t)$
        and $w_\alpha(z+\zeta,t)$,}\\
      0 & \text{otherwise.}
    \end{cases}
  \end{equation*}
  Thus for this choice
  \begin{equation*}
    D(z,t)=2\int_0^\infty \bigl(-\Phi_\alpha'\bigr)(\zeta)
    \abs{w_\alpha(z+\zeta,t)-w_\alpha(z,t)}
    \chi_{[m(z,\zeta),M(z,\zeta)]} (\zeta)\,d\zeta,
  \end{equation*}
  where $\chi_I$ denotes the indicator function of the interval $I$
  and
  \begin{equation*}
    m(z,\zeta)=\min\seq{w_\alpha(z+\zeta,t),w_\alpha(z,t)},
    \ \ \ 
    M(z,\zeta)=\max\seq{w_\alpha(z+\zeta,t),w_\alpha(z,t)}.
  \end{equation*}
\end{remark}

Next we demonstrate that the Lipschitz continuous weak solutions of the
filtered PDE \eqref{eq:intro-new-nonlocal-PDE} exhibit continuity with
respect to the initial data in the $L^1$ norm.  Specifically, we show
that the solution operator is $L^1$ contractive. 
It is important to note that solutions 
of \eqref{eq:walphapde} cannot be integrated over $\R$. 
However, the theorem below demonstrates that 
the difference between two solutions, if they 
are initially integrable, will be integrable 
over $\R$ at later times.

\begin{theorem}\label{thm:l1stability}
  Let $w_\alpha$ be a solution of \eqref{eq:walphapde} and let
  $v_\alpha$ be another solution with initial data $r_0$, 
  see Theorem \ref{thm:alphaconv}. If $y_0-r_0\in L^1(\R)$, then
  $w_\alpha(\cdot,t)-v_\alpha(\cdot,t) \in L^1(\R)$ 
  for $t>0$, and 
  \begin{equation*}
    \norm{w_\alpha(\cdot,t)-v_\alpha(\cdot,t)}_{L^1(\R)}\le
    \norm{y_0-r_0}_{L^1(\R)}.
  \end{equation*}
  In particular, Lipschitz continuous weak solutions are uniquely
  determined by their initial data.
\end{theorem}
\begin{proof}
 	Subtracting the equation for $v_\alpha$ 
 	from that of $w_\alpha$ we get
  	\begin{equation*}
    \pt \left(w_\alpha-v_\alpha\right)
    =\pz \left(\overline{W(w_\alpha)-W(v_\alpha)}\right).
  \end{equation*}
  Using the notation
  $\Delta W(z,t)=W(w_\alpha(z,t))-W(v_\alpha(z,t))$, we multiply this
  with $\sgn{w_\alpha(z,t)-v_\alpha(z,t)}=\sgn{\Delta W(z,t)}$ and get
  \begin{align}
    \pt \abs{w_\alpha-v_\alpha}
    &= \sgn{w_\alpha-v_\alpha} \pz
      \left(\overline{W(w_\alpha)-W(v_\alpha)}\right)\notag\\
    &=\int_0^\infty \Phi_\alpha'(\zeta) \sgn{\Delta W(z,t)}
      \left(\Delta W(z,t)-\Delta W(z+\zeta,t)\right)
      \,d\zeta\notag\\
    &\le \int_0^\infty \Phi_\alpha'(\zeta)  
      \left(\abs{\Delta W(z,t)}-\abs{\Delta W(z+\zeta),t}\right)
      \,d\zeta\notag\\
    &=\pz \int_0^\infty \Phi_\alpha(\zeta) \abs{\Delta W(z+\zeta,t)}\,d\zeta
    =\pz \overline{\abs{W(w_\alpha)-W(v_\alpha)}}\label{eq:Deltawin}.
  \end{align}
  Let $\delta>0$ be a constant, define $f_\delta(z)=e^{-\delta\abs{z}}$, 
  and observe that
  \begin{equation*}
    f'_\delta(z)=-\delta \sgn{z} f_\delta(z),\ \
    \abs{f'_\delta(z)}\le \delta f_\delta(z).
  \end{equation*}
  Multiply \eqref{eq:Deltawin} with $f_\delta(z)$ and integrate in $z$
  to get
  \begin{align*}
    \frac{d}{dt} \int_\R
    f_\delta(z)\abs{w_\alpha(z,t)-v_\alpha(z,t)}\,dz
    &\le -\int_\R f_\delta'(z)
      \overline{\abs{\Delta W}}(z,t)\,dz
      =\delta\int_\R \sgn{z}f_\delta(z)
      \overline{\abs{\Delta W}}(z,t)\,dz\\
    &\le \delta \int_0^\infty f_\delta(z)
      \overline{\abs{\Delta W}}(z,t)\,dz\\
    &= \delta\int_0^\infty \overline{f_\delta
      \abs{\Delta W}}(z,t)\,dz
      +\delta \int_0^\infty f_\delta(z)\overline{\abs{\Delta W}}(z,t)-
      \overline{f_\delta\abs{\Delta W}}(z,t)\,dz
    \\
    &\le \delta\norm{W'}_\infty\int_\R f_\delta(z)
      \abs{w_\alpha(z,t)-v_\alpha(z,t)}\,dz\\
    &\qquad
      +\delta\int_0^\infty \int_0^\infty \Phi_\alpha(\zeta)
      \left(f_\delta(z)-f_\delta(z+\zeta)\right)\abs{\Delta
      W(z+\zeta,t)}\,d\zeta dz\\
    &= \delta\norm{W'}_\infty\int_\R f_\delta(z)
      \abs{w_\alpha(z,t)-v_\alpha(z,t)}\,dz\\
    &\qquad
      +\delta\int_0^\infty \int_0^\infty \Phi_\alpha(\zeta)
      e^{-\delta z}\left(1-e^{-\delta\zeta}\right)\abs{\Delta
      W(z+\zeta,t)}\,d\zeta dz\\
    &\le  \delta\norm{W'}_\infty\int_\R f_\delta(z)
      \abs{w_\alpha(z,t)-v_\alpha(z,t)}\,dz\\
    &\qquad
      +M \int_0^\infty \Phi_\alpha(\zeta)
      \left(1-e^{-\delta\zeta}\right)\,d\zeta\\
    &\le  \delta\norm{W'}_\infty\int_\R f_\delta(z)
      \abs{w_\alpha(z,t)-v_\alpha(z,t)}\,dz
      +M \delta \int_0^\infty \Phi_\alpha(\zeta)\zeta\,d\zeta\\
    &=\delta\norm{W'}_\infty\int_\R f_\delta(z)
      \abs{w_\alpha(z,t)-v_\alpha(z,t)}\,dz
      +M \delta c \alpha,
  \end{align*}
  where $M$ is a bound on $\abs{\Delta W}$ and $c=\int_0^\infty
  \Phi(\zeta)\zeta\,d\zeta<\infty$, see \eqref{eq:kernel-ass1}.
  We invoke Gronwall's inequality and obtain
  \begin{align*}
     \int_\R f_\delta(z)
     \abs{w_\alpha(z,t)-v_\alpha(z,t)}\,dz
     &\le e^{\delta\norm{W'}_\infty t}\int_\R f_\delta(z)
       \abs{w_\alpha(z,0)-v_\alpha(z,0)}\,dz\\
      &\quad + \frac{M c\alpha}{\norm{W'}_\infty}
       \left(e^{\delta\norm{W'}_\infty t} -1\right).
  \end{align*}
  Since $w_\alpha(\cdot,0)-v_\alpha(\cdot,0)\in L^1(\R)$, we can use
  the monotone convergence theorem to take the limit as 
  $\delta\to 0$, and this concludes the proof.
\end{proof}

The following lemma presents three estimates that do not depend on the
parameter $\alpha$, and when taken together, they imply the local
$L^1$ precompactness of the sequence 
$\left\{w_\alpha\right\}_{\alpha>0}$. 
These estimates are modeled on
the discrete estimates from Corollary \ref{cor:bounds}.

\begin{lemma}
  \label{lem:walphabv}
  Let $w_\alpha$ be the unique Lipschitz continuous 
  solution of \eqref{eq:walphapde}, see Theorem \ref{thm:alphaconv}. 
  Then the following $\alpha$-independent estimates hold:
  \begin{align}
    \label{eq:walphabnd}
    & \inf_x y_0(z)\le w_\alpha(z,t)\le \sup_z y_0(z,t),
    \\ 
    \label{eq:walphabv}
    & \abs{w_\alpha(\cdot,t)}_{BV(\R)} \le \abs{y_0}_{BV(\R)},
    \\ 
    \label{eq:walphaL1cont}
    &\norm{w_\alpha(\cdot,t)-w_\alpha(\cdot,s)}_{L^1(\R)}\le
    \abs{t-s}\,\norm{W'}_\infty \abs{y_0}_{BV(\R)}.
  \end{align}
\end{lemma}
\begin{proof}
  Note the translation invariance of $\Phi_\alpha$ 
  in $\overline{\;\cdot\;}$, see the second part 
  of \eqref{eq:average-op}. Consequently, choosing 
  $v_\alpha(z,0)=w_\alpha(z+\zeta,0)$ in
  Theorem~\ref{thm:l1stability}, we conclude that 
  $\abs{w_\alpha(\cdot,t)}_{BV(\R)} \le
  \abs{w_\alpha(\cdot,0)}_{BV(\R)}\le \abs{y_0}_{BV(\R)}$.  
  This proves \eqref{eq:walphabv}.

  To prove \eqref{eq:walphaL1cont}, for $t>s$ we calculate
  \begin{align*}
    &\norm{w_\alpha(\cdot,t)-w_\alpha(\cdot,s)}_{L^1(\R)}
    \le \int_\R \int_s^t \abs{\overline{\pz W(w_\alpha)}(z,\tau)}\,d\tau\,dz
    \\ & \qquad 
    \le \int_s^t \int_\R \abs{\pz W(w_\alpha(z,\tau))}\,dz\,d\tau
    \\ & \qquad 
    \le \norm{W'}_\infty  \int_s^t
    \abs{w_\alpha(\cdot,\tau)}_{BV(\R)}\,d\tau
    \le (t-s)\norm{W'}_\infty \abs{y_0}_{BV(\R)}.
  \end{align*}

  It remains to prove \eqref{eq:walphabnd}. Let $a^+=\max\seq{a,0}$
  and $H(a)$ be the Heaviside function. By an approximation argument, 
  the functions
  \begin{equation*}
    \eta(w)=\left(w-k\right)^+,\quad 
    Q(w)=H(w-k)(W(w)-W(k)), \quad k\in \R,
  \end{equation*}
  are admissible entropy/entropy-flux pairs. Since $W$ is non-decreasing,
  $Q(w)=(W(w)-W(k))^+$. Using the notation of, and arguments
  similar to, the proof of Theorem~\ref{thm:l1stability} we find
  \begin{align*}
    \frac{d}{dt}\int_\R f_\delta(z) \eta(w_\alpha(z,t))\,dz
    &\le - \int_{0}^\infty f_\delta'(z) \overline{Q(w_\alpha)}(z,t)\,dz\\
    &= \delta\int_0^\infty \overline{f_\delta Q(w_\alpha)}(z,t)\,dz
      +\delta\int_0^\infty f_\delta(z)
      \overline{Q(w_\alpha)}(z,t)-\overline{f_\delta
      Q(w_\alpha)}(z,t)\,dz\\
    &\le \delta\int_\R f_\delta(z) Q(w_\alpha(z,t))\,dz\\
    &\qquad + \delta\int_0^\infty\int_0^\infty
      \Phi_\alpha(\zeta)\left(f_\delta(z)-f_\delta(z+\zeta)\right)
      Q(w_\alpha(z+\zeta,t))\,d\zeta\,dz\\
    &\le \delta\norm{W'}_\infty \int_\R f_\delta(z)
      \eta(w_\alpha(z,t))\,dz
    + M \delta c \alpha,
  \end{align*}
  where now $M$ is a bound on $Q$.
  Next, Gronwall's inequality yields
  \begin{align*}
    \int_\R f_\delta(z)
    \eta(w_\alpha(z,t))\,dz
    &\le e^{\delta\norm{W'}_\infty t}\int_\R f_\delta(z)
      \eta(w_\alpha(z,0))\,dz + \frac{M c\alpha}{\norm{W'}_\infty}
      \left(e^{\delta\norm{W'}_\infty t} -1\right).
  \end{align*}
  Thus if $w_\alpha(z,0)<k$ for almost all $z$ then
  \begin{equation*}
    \int_\R f_\delta(z)
    \eta(w_\alpha(z,t))\,dz \le
    \frac{M c\alpha}{\norm{W'}_\infty}
      \left(e^{\delta\norm{W'}_\infty t} -1\right),
  \end{equation*}
  for all $\delta>0$. We send $\delta\to 0$ and conclude that if
  $w_\alpha(z,0)<k$ for almost all $z$, then $w_\alpha(z,t)<k$ for
  almost all $z$. The other inequality is proved using
  $\eta(w)=(w-k)^-$ and analogous arguments.
\end{proof}

Consider now the scalar conservation law
\begin{equation}\label{eq:scalecons}
  \pt w = \pz W(w), \quad 
  w(\cdot,0)=y_0, 
  \quad z\in \R,\; t>0,
\end{equation}
which coincides with the original LWR equation \eqref{eq:LWR-eqn}
written in Lagrangian coordinates, where $W(\cdot)=V(1/w)$, see
\eqref{eq:wi-average-Wdef}, and $V$ is the local speed function. By a
solution of \eqref{eq:scalecons} we mean a 
distributional solution, i.e., a function 
$w=w(z,t)$ such that $w\in C([0,T];\lenloc(\R))
\cap L^\infty(\R\times[0,T])$, $T>0$, and
\begin{equation*}
  \int_0^T \int_\R w\pt \test - W(w)\pz \test\, dzdt =
  \int_\R w(z,T)\test(z,T) - y_0(z)\test(z,0)\,dz,
\end{equation*}
for all test functions $\test\in C^\infty_0(\R\times [0,T])$.

By an entropy solution of \eqref{eq:scalecons} we 
mean a weak solution which also satisfies
\begin{equation}
  \label{eq:entropyscalar}
  \int_0^T \int_\R \eta(w)\pt \test - Q(w)\pz \test\, dzdt \ge
  \int_\R \eta(w(z,T))\test(z,T) - \eta(y_0(z))\test(z,0)\,dz,
\end{equation}
for all entropy/entropy-flux pairs $(\eta,Q)$ and all non-negative
test functions in $\test\in C^\infty_0(\R\times [0,T])$.  If
$y_0\in BV(\R)$ (for example), there exists such unique entropy
solution $w$ of \eqref{eq:scalecons} \cite{Kruzkov:1970kx}.

By Lemma~\ref{lem:walphabv} the set $\seq{w_\alpha}_{\alpha>0}$ is
precompact in $C([0,T];L^1_{\operatorname{loc}}(\R))$, see
e.g.~\cite[Theorem~A.11]{Holden:2015aa}.  Let $\seq{\alpha}$ be some
subsequence such that $w=\lim_{\alpha\to 0} w_\alpha$ exists.

The following theorem demonstrates that the limit $w$ satisfies the
entropy inequalities, which identify the unique weak solution of
\eqref{eq:scalecons}. The fact that there is only one solution means
that the entire sequence $\seq{w_\alpha}$ converges to $w$, rather
than just a subsequence of it.

\begin{theorem}\label{thm:walphalimit} 
	Consider $W(\cdot)$ defined by \eqref{eq:wi-average-Wdef} 
	and an initial function $y_0\in BV(\mathbb{R})$ such 
	that $1\le y_0$. Suppose the averaging kernel $\Phi_\alpha$ 
	satisfies the conditions in \eqref{eq:kernel-ass1} 
	and \eqref{eq:kernel-ass2}. Then the limit 
	$w=\lim_{\alpha\to 0} w_\alpha$ coincides with the unique entropy
	solution to \eqref{eq:scalecons}.
\end{theorem}

\begin{proof}
  Let $\test$ be a non-negative test function and define
  \begin{align*}
    \Upsilon(w)
    &=\int_0^T \int_\R \eta(w)\pt \test - Q(w)\pz \test\, dzdt -
      \int_\R \eta(w(z,T))\test(z,T) - \eta(y_0(z))\test(z,0)\,dz,\\
    \Upsilon_\alpha(w)
    &=\int_0^T \int_\R \eta(w)\pt \test - \overline{Q(w)}\pz \test\, dzdt -
      \int_\R \eta(w(z,T))\test(z,T) - \eta(y_0(z))\test(z,0)\,dz.
  \end{align*}
  By Theorem~\ref{thm:entropyapriori}
  $\Upsilon_\alpha(w_\alpha)\ge 0$.  We write
  $\Upsilon(w)\ge
  \Upsilon_\alpha(w_\alpha)-\abs{\Upsilon_\alpha(w_\alpha)-\Upsilon(w)}\ge
  -\abs{\Upsilon_\alpha(w_\alpha)-\Upsilon(w)}$. Since
  $w_\alpha \to w$ in $C([0,T];L^1(\R))$, it is easily shown that
  $\abs{\Upsilon_\alpha(w_\alpha)-\Upsilon(w)}\to 0$ as $\alpha\to 0$.
  Hence the limit $w$ satisfies the entropy inequality
  \eqref{eq:entropyscalar} which implies that $w$ is a weak solution.
\end{proof}

We have shown that $w_\alpha(\cdot,t) \to w(\cdot,t)$ in
$L^1_{\operatorname{loc}}$ as $\alpha\to 0$.  By employing Kuznetsov's
lemma \cite[Theorem~3.14]{Holden:2015aa} we can demonstrate that
$w_\alpha\to w$ at a rate. For simplicity, we assume that
$\lim_{\abs{z}\to \infty} y_0(z)=c$ for some constant $c$. Since
$v_\alpha=c$ is a solution of \eqref{eq:walphapde},
Theorem~\ref{thm:l1stability} ensures that
$w_\alpha(\cdot,t)-c\in L^1(\R)$. Since $w$ solves the scalar
conservation law \eqref{eq:scalecons}, by finite speed of propagation,
$w(\cdot,t)-c \in L^1(\R)$ and thus
$w_\alpha(\cdot,t)-w(\cdot,t)\in L^1(\R)$.  To state Kuznetsov's
lemma, we need some notation.  Let $(\eta,Q)$ be the Kru\v{z}kov
entropy/entropy-flux pair
\begin{equation*}
  \eta(w)=\abs{w-k},
  \quad 
  Q(w,k)=\abs{W(w)-W(k)},
\end{equation*}
and let
\begin{equation*}
  \begin{aligned}
    \Lambda_T(w,\test,k)&=\int_0^T \int_\R \eta(w(z,t))
    \pt \test(z,t)-
    Q(w(z,t),k)\pz \test(z,t)\, dzdt \\
    &\qquad - \int_\R \eta(w(z,T))\test(z,T) -
    \eta(y_0(z))\test(z,0)\,dz.
  \end{aligned}
\end{equation*}
Let $\omega_\eps$ be a standard mollifier and define the test function
\begin{equation*}
  \Omega_{\eps_0,\eps}(z,z',t,t')=\omega_{\eps_0,\eps}(t-t')\omega_\eps(z-z').
\end{equation*}
Let $w_\alpha$ be the unique solution of \eqref{eq:walphapde} and let
$w$ be the entropy solution of \eqref{eq:scalecons}. Observe that $w$
and $w_\alpha$ share the same initial data. Finally define
\begin{equation*}
  \Lambda_{\eps_0,\eps}(w_\alpha,w)=
  \int_0^T\int_\R \Lambda_T
  \left(w_\alpha,\Omega\left(\cdot,t',\cdot,z'\right),w\left(z',t'\right)
  \right)\,dz'\,dt'.
\end{equation*}
Since we know that $\abs{w(\cdot,t)}_{BV(\R)}\le \abs{y_0}_{BV(\R)}$
and $\abs{w_\alpha(\cdot,t)}_{BV(\R)}\le \abs{y_0}_{BV(\R)}$, in this
context Kuznetsov's lemma reads
\begin{equation*}
  \norm{w_\alpha(\cdot,t)-w(\cdot,t)}_{L^1(\R)} \le
  2\left(\eps+\norm{W'}_\infty\eps_0\right)\abs{y_0}_{BV(\R)}-
  \Lambda_{\eps_0,\eps}(w_\alpha,w).
\end{equation*}
This can be used to prove the following result quantifying the
convergence $w_\alpha \to w$.

\begin{theorem}\label{lem:walpharate}
  Suppose the assumptions of Theorem \ref{thm:walphalimit} hold. 
  Let $w_\alpha$ and $w$ be solutions
  respectively of \eqref{eq:walphapde} and \eqref{eq:scalecons}. Then
  \begin{equation*}
    \norm{w_\alpha(\cdot,t)-w(\cdot,t)}_{L^1(\R)}
    \le 2\sqrt{2 T \norm{W'}_\infty\abs{y_0}_{BV(\R)}\alpha}, \
    \ \ \text{for $t\le T$.}
  \end{equation*}
\end{theorem}

\begin{proof}
  Using Theorem~\ref{thm:entropyapriori}
  \begin{align*}
    -\Lambda_{\eps_0,\eps}(w_\alpha,w)
    =&-\overline{\Lambda}_{\eps_0,\eps}(w,w_\alpha) +
       \overline{\Lambda}_{\eps_0,\eps}(w,w_\alpha) -
       \Lambda_{\eps_0,\eps}(w_\alpha,w)\\
     &\le \abs{\overline{\Lambda}_{\eps_0,\eps}(w,w_\alpha) -
       \Lambda_{\eps_0,\eps}(w_\alpha,w)},
  \end{align*}
  where
  \begin{align*}
    \overline{\Lambda}_{\eps_0,\eps}(w_\alpha,w)
    &=
      \int_0^T\int_\R\int_0^T\int_\R
      \abs{w_\alpha(z,t)-w(z',t')} \pt\Omega(z,z',t,t')\\
    &\quad\hphantom{\int_0^T\int_\R\int_0^T\int_\R}
      -\overline{Q(w_\alpha,w(z',t'))}(z,t)\pz\Omega(z,z',t,t')
      \,dzdt\,dz'dt'\\
    &\qquad -
      \int_0^T\int_\R\int_\R \abs{w_\alpha(z,T)-w(z',t')}\Omega(z,z',T,t')\\
    &\qquad \hphantom{-\int_0^T\int_\R\int_\R}
      -\abs{w_\alpha(z,0)-w(z',t')}\Omega(z,z',0,t')\,dz\,dz'dt'.
  \end{align*}
Thus
\begin{align*}
  \overline{\Lambda}_{\eps_0,\eps}(w,w_\alpha) -
  \Lambda_{\eps_0,\eps}(w_\alpha,w)
  &=\int_0^T\int_\R\int_0^T\int_\R
    \left(Q(w_\alpha(z,t),w(z',t'))-\overline{Q(w_\alpha,w(z',t'))}(z,t)\right)
    \\
  &\hphantom{\int_0^T\int_\R\int_0^T\int_\R}\quad\times
    \pz\Omega(z,z',t,t')
    \,dzdt\,dz'dt'.
\end{align*}
Regarding the difference $Q(\ )-\overline{Q}(\ )$,
\begin{align*}
  \abs{Q(w_\alpha(z,t),w(z',t'))-\overline{Q(w_\alpha,w(z',t'))}(z,t)}
  &=\Bigl|
    \int_0^\infty \Phi_\alpha(\zeta)
    \bigl(Q(w_\alpha(z,t),w(z',t'))\\
  &\hphantom{=\Bigl|
    \int_0^\infty \Phi_\alpha(\zeta)
    \bigl(} \quad
    -Q(w_\alpha(z+\zeta,t),w(z',t'))\bigr)\,d\zeta\Bigr|\\
  &\le \norm{W'}_\infty
    \int_0^\infty \Phi_\alpha(\zeta)
    \abs{w_\alpha(z+\zeta,t)-w_\alpha(z,t)}\,d\zeta.
\end{align*}
Therefore we can proceed as follows:
\begin{align*}
  -\Lambda_{\eps_0,\eps}(w_\alpha,w)
  &\le  \norm{W'}_\infty\int_0^T\int_\R\int_0^T\int_\R \int_0^\infty
    \Phi_\alpha(\zeta) \abs{w_\alpha(z+\zeta,t)-w_\alpha(z,t)}\\
  &\hphantom{\le  \norm{W'}_\infty\int_0^T\int_\R\int_0^T\int_\R
    \int_0^\infty}
    \times \omega_{\eps_0}(t-t')\omega_{\eps}'(z-z')
    \,d\zeta \,dzdt\,dz'dt\\
  &\le  \norm{W'}_\infty\int_0^T  \int_0^\infty
    \Phi_\alpha(\zeta)\zeta \abs{y_0}_{BV(\R)} \frac{1}{\eps}
    \,d\zeta\,dt\\
  &\le T \norm{W'}_\infty\abs{y_0}_{BV(\R)} \frac{\alpha}{\eps},
\end{align*}
where we have used \eqref{eq:kernel-ass1}. Hence
\begin{equation*}
   \norm{w_\alpha(\cdot,t)-w(\cdot,t)}_{L^1(\R)} 
   \le 2\eps + T \norm{W'}_\infty\abs{y_0}_{BV(\R)} \frac{\alpha}{\eps},
 \end{equation*}
 for $\eps>0$. Minimising the right hand side over $\eps$ concludes
 the proof. 
\end{proof}

Theorems \ref{thm:walphalimit} and \ref{lem:walpharate} 
state that as the filter size $\alpha$ approaches 0, the 
filtered variables $w_\alpha$, which are equal 
to $\overline{y_\alpha}$, converge strongly 
in $L^1_{\operatorname{loc}}$ to the entropy 
solution of the LWR conservation law \eqref{eq:scalecons}. 
By Proposition \ref{prop:weak-conv-yalpha}, we 
know only that $y_\alpha$ converges weakly. 
The question of whether the 
Lagrangian variables $y_\alpha$ (spacing between cars) 
also converge strongly is a natural one, and our 
next result shows that this is true when 
using the exponential kernel. 

\begin{corollary}\label{lem:strongy}
  Suppose the assumptions of Theorem \ref{thm:walphalimit} hold, 
  and specify $\Phi(\zeta)=e^{-\zeta}$. 
  Let $y_\alpha$ and $w$ be solutions respectively 
  of \eqref{eq:3c} and \eqref{eq:scalecons}. Then
  \begin{equation*}
    \norm{y_\alpha(\cdot,t)-w(\cdot,t)}_{L^1(\R)}
    \le \alpha\abs{y_0}_{BV(\R)}
    +2\sqrt{2 T \norm{W'}_\infty\abs{y_0}_{BV(\R)}\alpha}, \
    \quad \text{for $t\in [0,T]$.}
  \end{equation*}
\end{corollary}

\begin{proof}
Due to the special choice of the function $\Phi$ we have the identity 
$-\alpha\pz w_\alpha+w_\alpha=y_\alpha$.
Thus, using \eqref{eq:walphabv} and 
Theorem \ref{lem:walpharate}, we get
\begin{align*}
    \norm{y_\alpha(\cdot,t)-w(\cdot,t)}_{L^1(\R)}
    & \le  \norm{y_\alpha(\cdot,t)-w_\alpha(\cdot,t)}_{L^1(\R)}
    +\norm{w_\alpha(\cdot,t)-w(\cdot,t)}_{L^1(\R)}
    \\ &
    \le \alpha \abs{w_\alpha(\cdot,t)}_{BV(\R)}
    +2\sqrt{2 T \norm{W'}_\infty\abs{y_0}_{BV(\R)}\alpha}
    \\ &
    \le \alpha  \abs{y_0}_{BV(\R)}
    +2\sqrt{2 T \norm{W'}_\infty\abs{y_0}_{BV(\R)}\alpha}.
\end{align*}
\end{proof}

\begin{remark}
Let us examine conditions on the 
kernel $\Phi$ that enhance the weak convergence of ${y_\alpha}$ 
from Proposition \ref{prop:weak-conv-yalpha} to strong convergence (to the 
limit $w$ of $w_\alpha$). It appears that the only 
scenario is the one described in Corollary \ref{lem:strongy}.
Using \eqref{eq:3c},
\begin{align*}
	\pz w_{\alpha}(z,t) & =-\Phi_\alpha(0)y_{\alpha}(z,t)
	-\int_{0}^{\infty}\Phi_\alpha'\left({\zeta}\right)
	y_{\alpha}(\zeta+z,t)\,d\zeta
	\\ & =\Phi_\alpha(0)(w_\alpha(z,t)-y_{\alpha}(z,t))
  -\int_{0}^{\infty}\left(\Phi_\alpha(0)\Phi_\alpha(\zeta)
  +\Phi_\alpha'\left({\zeta}\right)\right)y_{\alpha}(\zeta+z,t)\,d\zeta
  \\ & =
  \frac{\Phi(0)}{\alpha}(w_\alpha(z,t)-y_{\alpha}(z,t))
  -\frac{1}{\alpha^2}\int_{0}^{\infty}
  \left(\Phi(0)\Phi\left(\frac{\zeta}{\alpha}\right)
  +\Phi'\left(\frac{\zeta}{\alpha}\right)\right)
  y_{\alpha}(\zeta+z,t)\,d\zeta.
\end{align*}
For every $R>0$, using \eqref{eq:4} and \eqref{eq:walphabv},
\begin{align*}
	\int_{-R}^R & \abs{w_\alpha(z,t)-y_{\alpha}(z,t)}\, dz
	\\ & 
	\le \frac{\alpha}{\Phi(0)}\int_{-R}^R 
	\abs{\pz w_{\alpha}(z,t)}\, dz
	+\frac{1}{\alpha\Phi(0)}\int_{-R}^R\int_{0}^{\infty}
	\left|\Phi(0)\Phi\left(\frac{\zeta}{\alpha}\right)
	+\Phi'\left(\frac{\zeta}{\alpha}\right)\right|y_{\alpha}(\zeta+z,t)
	\,d\zeta\,dz
	\\ & \le 
	\frac{\alpha}{\Phi(0)}\abs{w_\alpha(\cdot,t)}_{BV(\R)}
	+\frac{2R\norm{y_\alpha(\cdot,t)}_{L^\infty(\R)}}{\Phi(0)}
	\int_{0}^{\infty}\left|\Phi(0)\Phi\left({\zeta}\right)
	+\Phi'\left({\zeta}\right)\right|\,d\zeta
	\\ & \le 
	\frac{\alpha}{\Phi(0)}\abs{y_0}_{BV(\R)}
	+\frac{2R\norm{y_0}_{L^\infty(\R)}}{\Phi(0)}
	\int_{0}^{\infty}\left|\Phi(0)\Phi\left({\zeta}\right)
	+\Phi'\left({\zeta}\right)\right|\,d\zeta.
\end{align*}
Strong convergence is achieved only when the last term is zero, 
meaning $\Phi(0)\Phi\left({\zeta}\right)+\Phi'\left({\zeta}\right)=0$, which 
only holds when $\Phi(\zeta)=e^{-\zeta}$. Although numerical 
evidence suggests that strong convergence of $y_\alpha$ occurs 
for Lipschitz continuous kernels different from $e^{-\zeta}$, weak 
convergence (oscillations persist) is observed for $BV$ (discontinuous) 
kernels in the limit as $\alpha\to 0$.
\end{remark}

\section{Numerical examples}\label{sec:numerics}
\graphicspath{{figures/}}

This section presents three numerical experiments 
that showcase the features of our proposed model and 
compare it with established models in the field, giving a 
deeper understanding and valuable insights for future improvement.

\subsection{Comparing different models}

We compare solutions of the standard (local) 
LWR FtL model, the more sophisticated non-local FtL model given by
\eqref{eq:nonlocal-FtL-density}, \eqref{eq:arithmetic-mean}, and the
nonlocal FtL model \eqref{eq:nonlocal-FtL-harmonic},
\eqref{eq:Lagrangeweights} proposed in this work. 

Concretely, let the initial values (initial
positions of vehicles) $x_i(0)=\tx_i(0)=\bx_i(0)$ be specified as
follows: Let $\ell$ be a small parameter (the length of a vehicle) and
$\rho_0$ be a function such that $0<\rho_0(x)\le 1$ and that
$\rho_0(x)$ is constant for $x$ outside the interval $(a,b)$. Then we
set $x_1(0)=a$ and define $x_{i+1}(0)$, $u_i(0)$ by
\begin{equation*}
  \int_{x_i(0)}^{x_{i+1}(0)} \rho_0(x)\,dx=\ell,
  \quad
  u_i(0)=\frac{\ell}{x_{i+1}(0)-x_i(0)}, 
  \quad 
  i=1,\ldots,N,
\end{equation*}
where $N$ is the smallest integer such that
$x_{N+1}(0)>b$. Finally, we set $u_{N+1}(0)=\rho_0(x_{N+1})$,
$x_{N+1}=\infty$ and $u_0(0)=\rho_0(a-1)$. Given
$\seq{\tx_i}_{i=1}^{N+1}$ with $x_{N+1}=\infty$ and $z_i=i\ell$ for
$i=1,\ldots,N$, $z_{N+1}=\infty$, define the $N\times N$ upper
triangular matrices $\tilde{\Phi}_\alpha$ and $\overline{\Phi}_\alpha$
with entries
\begin{equation*}
  \tilde{\Phi}_{i,j,\alpha}=\int_{\tx_j}^{\tx_{j+1}}
  \Phi_\alpha\left(\xi-\tx_i\right)\,d\xi, \ \ \ \
  \overline{\Phi}_{i,j,\alpha}=\int_{z_j}^{z_{j+1}}
  \Phi_\alpha\left(\zeta-z_i\right)\,d\zeta,    
\end{equation*}
respectively. Observe that
$\tilde{\Phi}_{N,N,\alpha}=\overline{\Phi}_{N,N,\alpha}=1$. For $t>0$,
$i=1,\ldots,N-1$, let $x_i(t)$, $\tx_i(t)$ and $\bx_i(t)$ solve
\begin{align}
  \text{(local FtL)}
  \qquad 
  x_i'
  &=V(u_i),\ \ \
    u_i = \frac{\ell}{x_{i+1}-x_i},
    \label{eq:numLWR}\\
  \text{(standard nonlocal FtL)}
  \qquad
  \tx_i'
  &=V(\tu_i),\ \ \
    \tu_i = \sum_{j=i}^N\tilde{\Phi}_{i,j,\alpha}
    \frac{\ell}{\tx_{j+1}-\tx_j},
    \label{eq:numtilde}\\
  \text{(our nonlocal FtL)}
  \qquad
  \bx_i'
  &=V(\bu_i),\ \ \
    \bu_i = \Bigl(\, \sum_{j=i}^N
    \overline{\Phi}_{i,j,\alpha}
    \frac{\bx_{i+1}-\bx_i}{\ell}\, \Bigr)^{-1},
    \label{eq:numbar}
\end{align}
and $x_N'=\tx_N'=\bx_N'=V(u_N)$, where $V$ is a non-increasing 
Lipschitz continuous function $V:[0,1]\mapsto [0,1]$ with $V(1)=0$. 
We define the piecewise constant
function
\begin{equation*}
  u_\ell(x,t)=
  \begin{cases}
    u_0 & x\le a,\\ u_i(t)& x_{i}(t)< x\le x_{i+1}(t), 
    \ \
    i=1,\ldots,N-1,\\
    u_{N}& x_N(t) < x.
  \end{cases}
\end{equation*}
The piecewise constant functions $\tu_\ell$ and $\bu_\ell$ are defined
analogously. To solve \eqref{eq:numLWR} -- \eqref{eq:numbar}
numerically we utilise the explicit Euler scheme with $\Dt=\ell$. In
all our computations we use
\begin{equation*}
  V(v)=1-v\ \ \ \text{and}\ \ \ \Phi(\xi)=e^{-\xi}.
\end{equation*}
We consider the (box) initial condition
\begin{equation}
  \label{eq:rho01}
  \rho_0(x)=
  \begin{cases}
    1 & \abs{x}<0.75,\\ 0.05 &\text{otherwise.}
  \end{cases}
\end{equation}
If Figure~\ref{fig:alphalarge} we show a numerical solution to
\eqref{eq:numLWR} -- \eqref{eq:numbar} computed with the explicit
Euler scheme and $\alpha=0.5$ at $t=1.4$ for $\ell=0.06$ (left) and
$\ell=0.005$ (right). It appears that the limits as $\ell\to 0$ of
$\tu_\ell$ and $\bu_\ell$ are different, and that both of these differ
from the limit of $u_\ell$ --- the entropy solution of
the conservation law \eqref{eq:LWR-eqn}. We also observe that the
limits of $\tu_\ell$ and $\bu_\ell$ (as $\ell \to 0$) seem to have both
positive and negative jumps and thus cannot satisfy an Oleinik type
entropy condition.
\begin{figure}[h]
  \centering
  \begin{tabular}[h!]{lr}
    \includegraphics[width=0.5\linewidth]{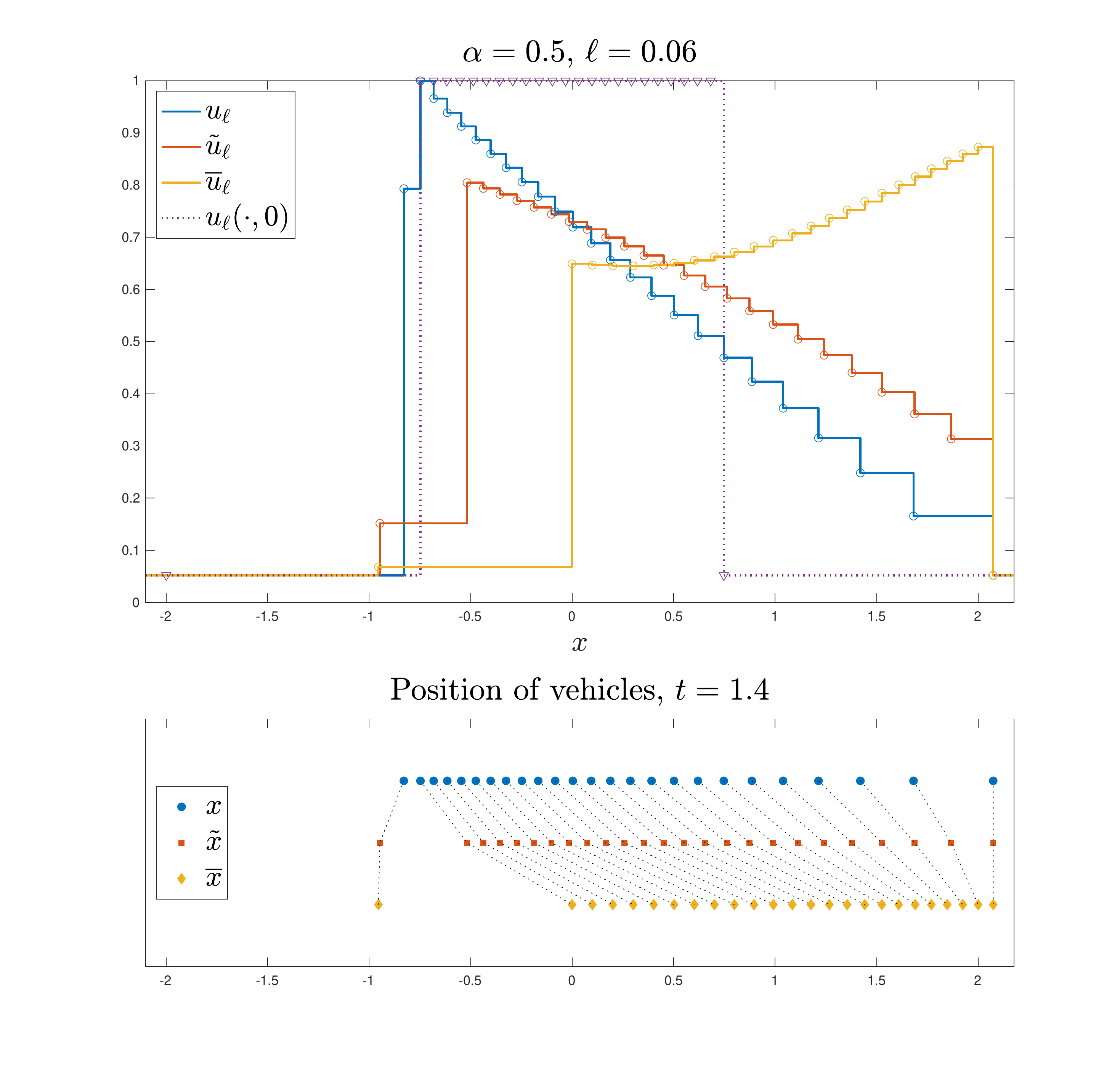}
    &\includegraphics[width=0.5\linewidth]{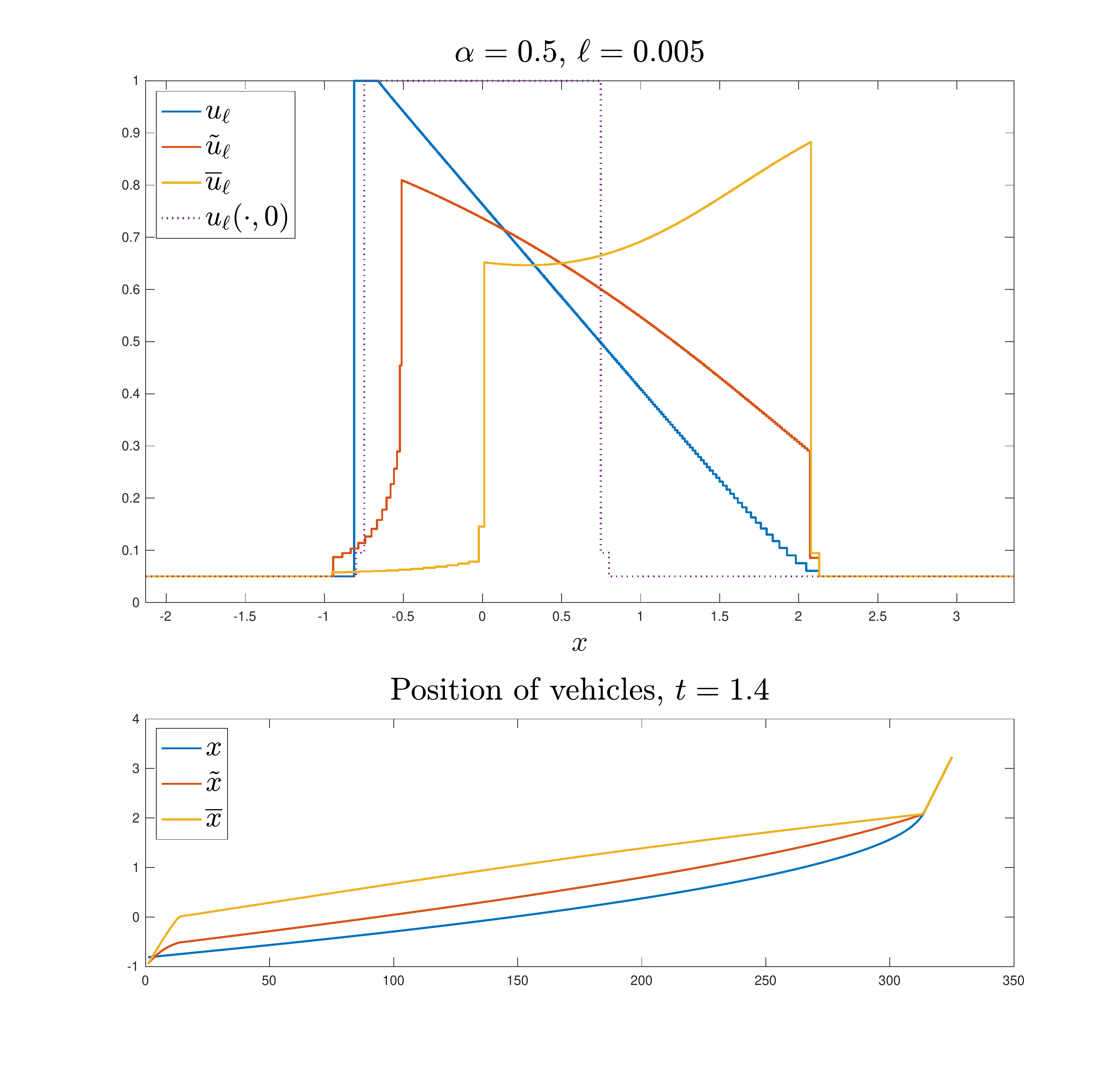}
  \end{tabular}
  \caption{Numerical solutions of \eqref{eq:numLWR} --
    \eqref{eq:numbar} computed by the explicit Euler scheme. Left:
    $\ell=0.06$ , right: $\ell=0.005$.}
  \label{fig:alphalarge}
\end{figure}

The simulations show that when the speed is determined 
using weighted Lagrangian coordinates \eqref{eq:numbar}, vehicles drive faster compared 
to when the speed is determined by the local FtL model \eqref{eq:numLWR} 
or Eulerian coordinates \eqref{eq:numtilde}. This is because 
the Lagrangian distance between vehicles remains constant 
even if the Eulerian distance increases. The Lagrangian 
distance is always less than or equal to the Eulerian distance, giving 
the Lagrangian model more weight to spacings further ahead. 
As a result, in a decreasing density or thinly occupied road, the 
speed determined by the Lagrangian model is greater 
than or equal to that determined by the Eulerian model.

\subsection{The zero-filter ($\alpha\to 0$) limit}\label{sub:ato0}

We now study the scheme \eqref{eq:wscheme} for $\alpha=1/2$,
$\alpha=1/8$, $\alpha=1/32$ and $\alpha=1/128$ in order to compare
$1/w_\alpha$ and $1/y_\alpha$ with $\rho$, where $\rho$ is the unique 
entropy solution of the local LWR model
\begin{equation}\label{eq:rholaw}
  \pt\rho+\px (\rho V(\rho))=0,\ \ \ \rho(x,0)=\rho_0(x).
\end{equation}
In this setting ($\rho_0=\mathrm{const}$ outside an interval $(a,b)$),
we define $u_i(0)$ and the matrix $\overline{\Phi}_\alpha$ as in the
previous section and then define the initial data
\begin{equation}
  \label{eq:initdata}
  y_i^0=\frac{1}{u_i(0)}\ \ \ \text{and}\ \ \ w^0_i = \sum_{j=i}^N
  \overline{\Phi}_{i,j,\alpha} y^0_j,
\end{equation}
for $i=1,\ldots,N$. Set $\Dz=\ell$, $\lambda=\Dt/\Dx$ where $\Dt$ is
chosen such that the CFL-condition \eqref{eq:CFLcond} holds. Let
$w^n_i$ satisfy \eqref{eq:wscheme}, which in this context reads
\begin{equation}
  \label{eq:wschemep}
  w_i^{n+1}=w^n_i+\lambda\Bigl(\sum_{j=i+1}^N
  \overline{\Phi}_{i+1,j,\alpha} W\left(w^n_j\right)-
  \sum_{j=i+1}^N
  \overline{\Phi}_{i,j,\alpha} W\left(w^n_j\right) \Bigr),
\end{equation}
for $i=1,\ldots,N$. The scheme for $y^n_i$ then reads
\begin{equation*}
  y_i^{n+1}=y^n_i+\lambda\left(W\left(w^n_{i+1}\right)
  -W\left(w^n_{i}\right)\right),
\end{equation*}
for $i=1,\ldots,N$. It is not very elucidating to compare $1/y$ and
$1/w$ with $\rho$ in Lagrangian coordinates, let therefore the
``discrete Eulerian coordinates'' $\xi_i^n$ be defined by
\begin{equation*}
  \xi^n_1=x_1(0)+\Dt \sum_{m=1}^n 
  V\Bigl(\,\frac{1}{y^m_1}\,\Bigr),
  \ \ \
  \xi^n_{i+1}=\xi^n_i+y^n_i\Dz,\ \ i=1,\ldots,N-1,
\end{equation*}
cf.~\eqref{eq:uidef}. Hence, we expect that
\begin{equation*}
  \frac{1}{w^n_i}\approx \rho(\xi^n_i,t^n)\ \ \ \text{and}\ \ \
  \frac{1}{y^n_i}\approx \rho(\xi^n_i,t^n)
\end{equation*}
for sufficiently small $\alpha$.
\begin{figure}[h]
  \centering
  \begin{tabular}[h!]{lr}
    \includegraphics[width=0.5\linewidth]{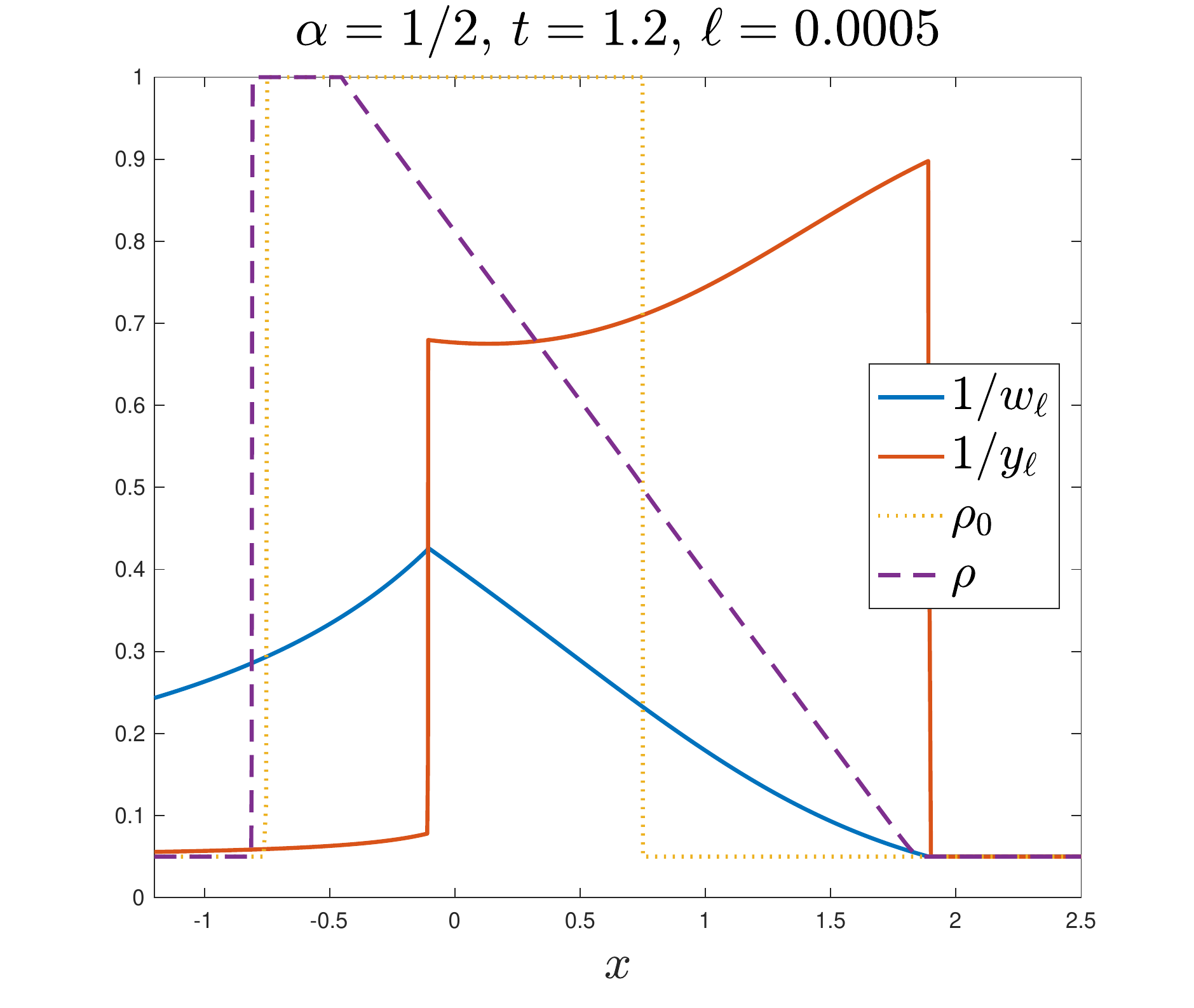}
    &\includegraphics[width=0.5\linewidth]{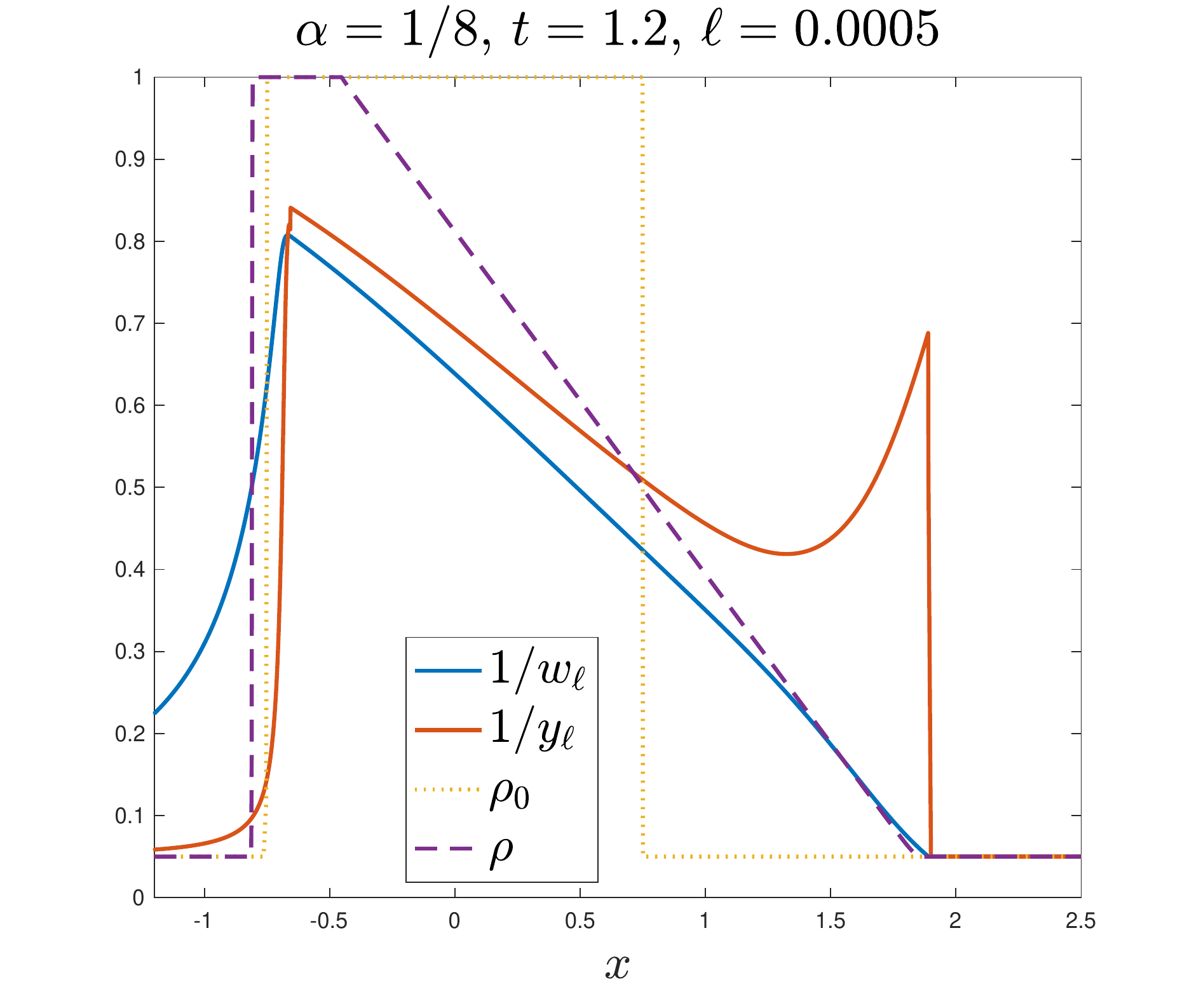}\\
    \includegraphics[width=0.5\linewidth]{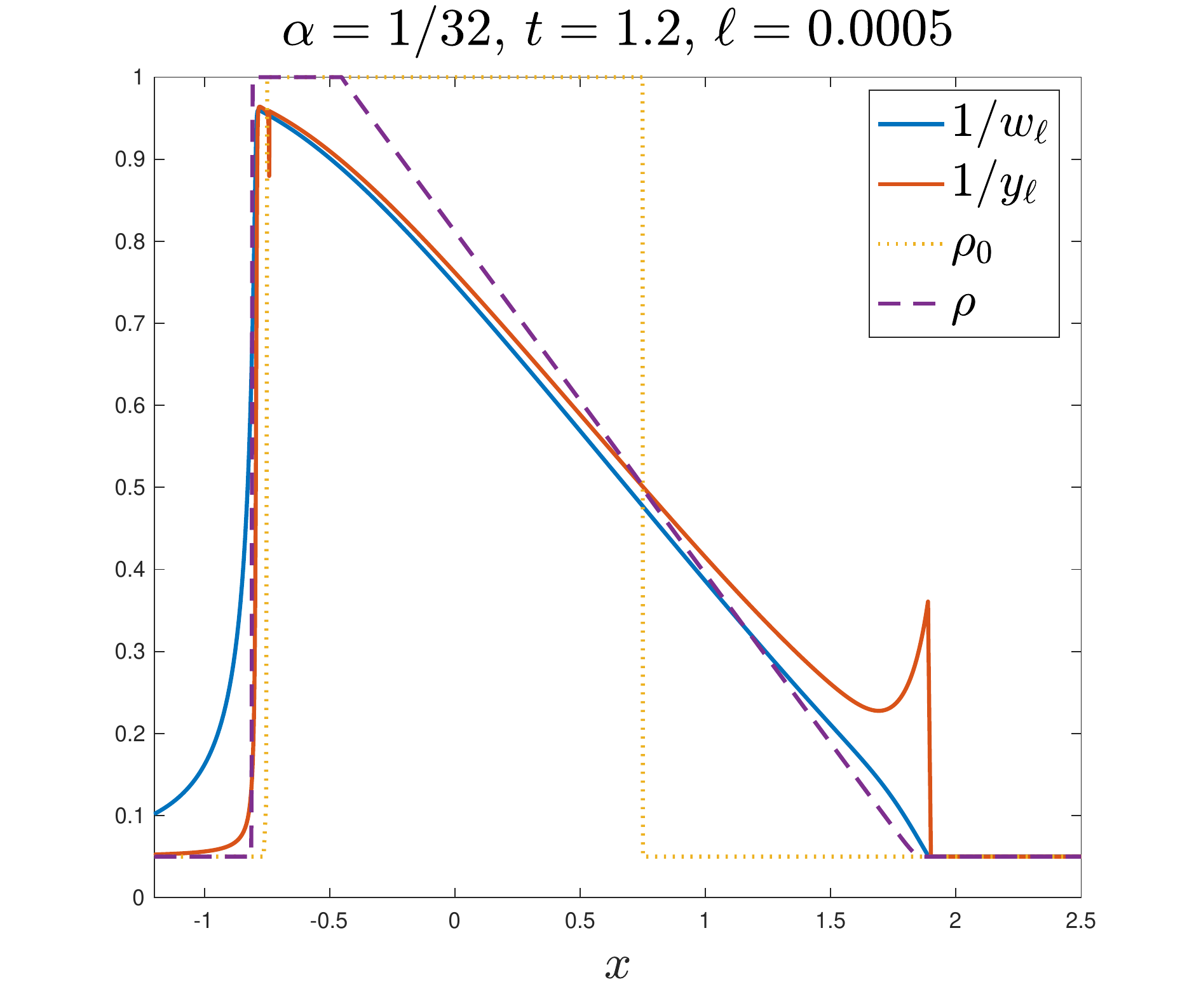}
    &\includegraphics[width=0.5\linewidth]{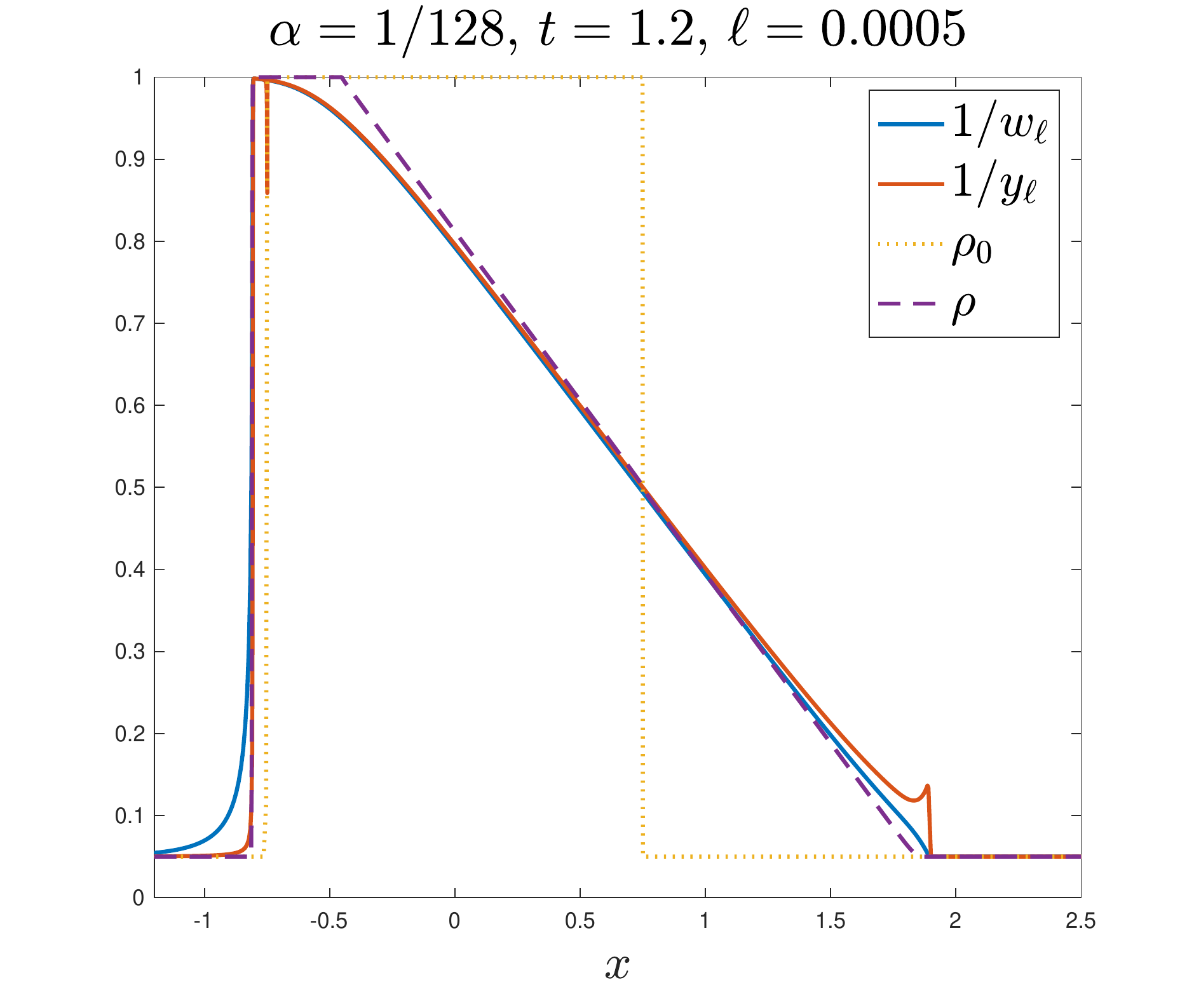}    
  \end{tabular}
  \vspace{-1.5em}
  \caption{Solutions of \eqref{eq:wschemep}, \eqref{eq:numbar},with
    initial data given in \eqref{eq:rho01}, \eqref{eq:initdata}. For
    all computations $t=1.2$ and $\ell=1/2000$. For comparisons we
    also show a numerical solution of \eqref{eq:rholaw}. Upper left:
    $\alpha=1/2$, upper right: $\alpha=1/8$, lower left:
    $\alpha=1/32$, lower right: $\alpha=1/128$.}
  \label{fig:alphaconv}
\end{figure}
In Figure~\ref{fig:alphaconv} shows $1/w$, $1/y$ and $\rho$ for
different values of $\alpha$. In these plots the $x$ axis is the Eulerian
coordinates, i.e., we plot the points
\begin{equation*}
  \left(\xi^n_i,1/w^n_i\right)\ \ \ \text{and}\ \ \
  \left(\xi^n_i,1/y^n_i\right),
\end{equation*}
for all relevant $i$, and $n$ is such that $t^n=1.2$. The
approximation to the conservation law \eqref{eq:rholaw} is computed
with the Engquist-Osher scheme on a fine grid. From this figure, we
see that $\seq{1/w^n_i}_{i=1}^N$ and $\seq{1/y^n_i}_{i=1}^N$ approach
$\seq{\rho(\xi^n_i,t^n)}_{i=1}^N$ in $L^1$ as $\alpha$ traverses the
sequence $\seq{1/2,1/8,1/32,1/128}$.

\subsection{Convergence of $y_\alpha$ and the effect of different filters.}
\label{subsec:strongconv} 

We proved that the filter $\Phi=\Phi_{\mathrm{exp}}(z)=e^{-z}$ 
results in strong convergence of $y_\alpha$ to $1/\rho$, the 
entropy solution of the local LWR conservation law \eqref{eq:LWR-eqn}. 
This convergence, which followed from  
$\norm{y_\alpha(\cdot,t)-w_\alpha(\cdot,t)}_{L^1(\R)} 
\lesssim \mathcal{O}(\alpha)$, was also seen in 
previous experiments. However, this strong convergence 
has only been proven for this specific filter 
and may not hold for others. To test this we
experimented with other Lipschitz continuous filters:
\begin{equation*}
	\Phi_1(z)=\frac{4}{\pi}\frac{1}{(1+z^2)^2},\ \ \Phi_{\mathrm{tri}}(z)=2\max\seq{1-z,0}\ \
   \text{and even}\ \ \Phi_2(z)=\frac{2}{\pi}\frac{1}{1+z^2},
\end{equation*}
although the last filter is not covered by the theory in this
paper. Our numerical experiments show that $y_\alpha$ 
converges strongly for all filters. However, for the 
discontinuous filter $\Phi_{\mathrm{box}}(z)=\chi_{(0,1)}(z)$, we 
observe weak convergence oscillations that persist 
as $\alpha\to 0$. 

Oscillatory solutions can be attributed 
to stop-and-go traffic patterns \cite{TreiberKesting:13}. Recall that 
stop-and-go traffic refers to a situation where cars 
frequently start and stop, resulting in waves of congestion 
that can propagate through a traffic flow and cause oscillations. 

In Figure~\ref{fig:oscillations} we compare computations
using the initial data \eqref{eq:rho01}, $\ell=1/5000$, and the
filters $\Phi_{\mathrm{tri}}$ (left column) and $\Phi_{\mathrm{box}}$
(right column). In the first row $\alpha=1/32$ and in the second row
$\alpha=1/128$.
\begin{figure}[h]
  \centering
  \begin{tabular}[h!]{lr}
    \includegraphics[width=0.5\linewidth]{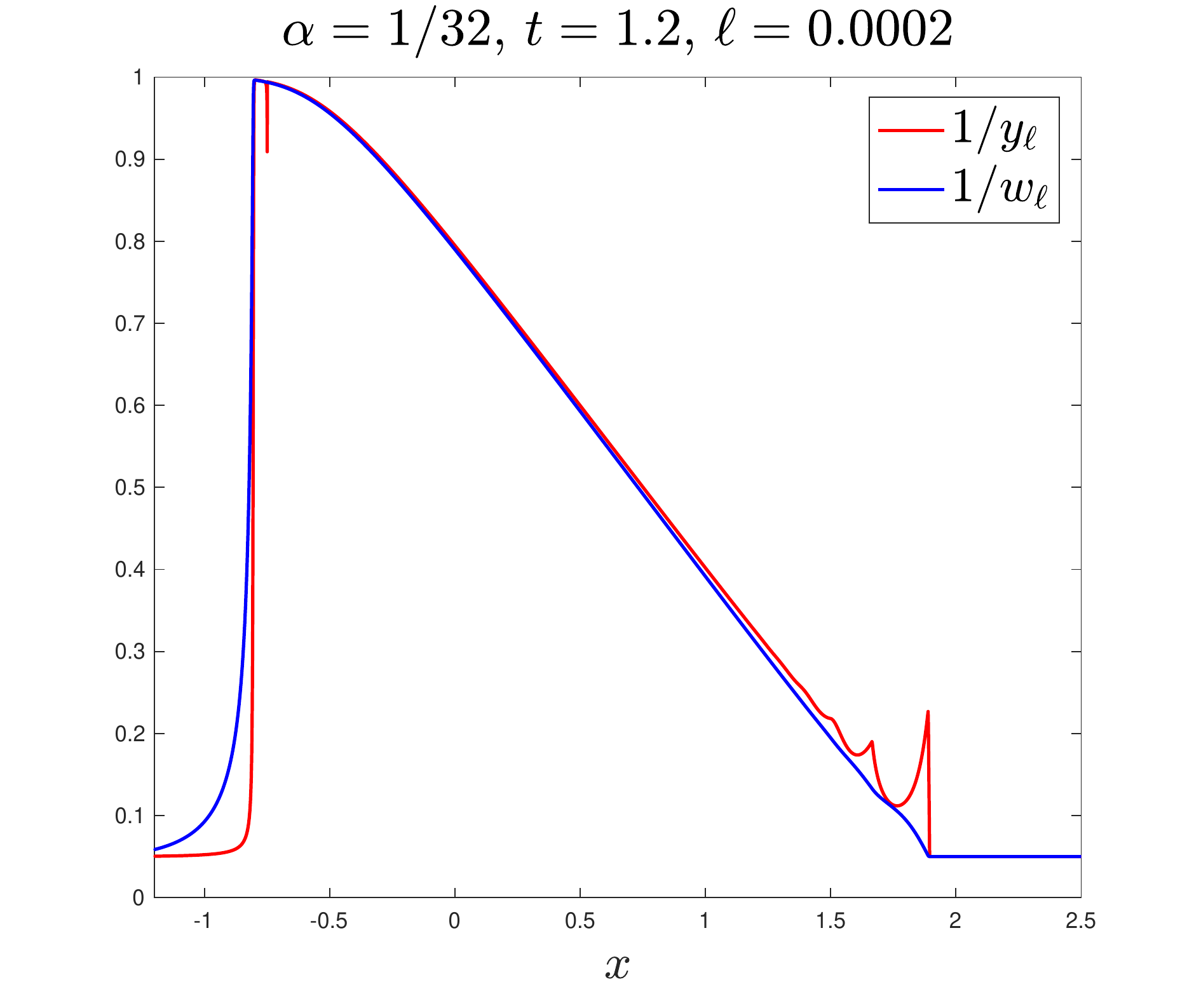}
    &\includegraphics[width=0.5\linewidth]{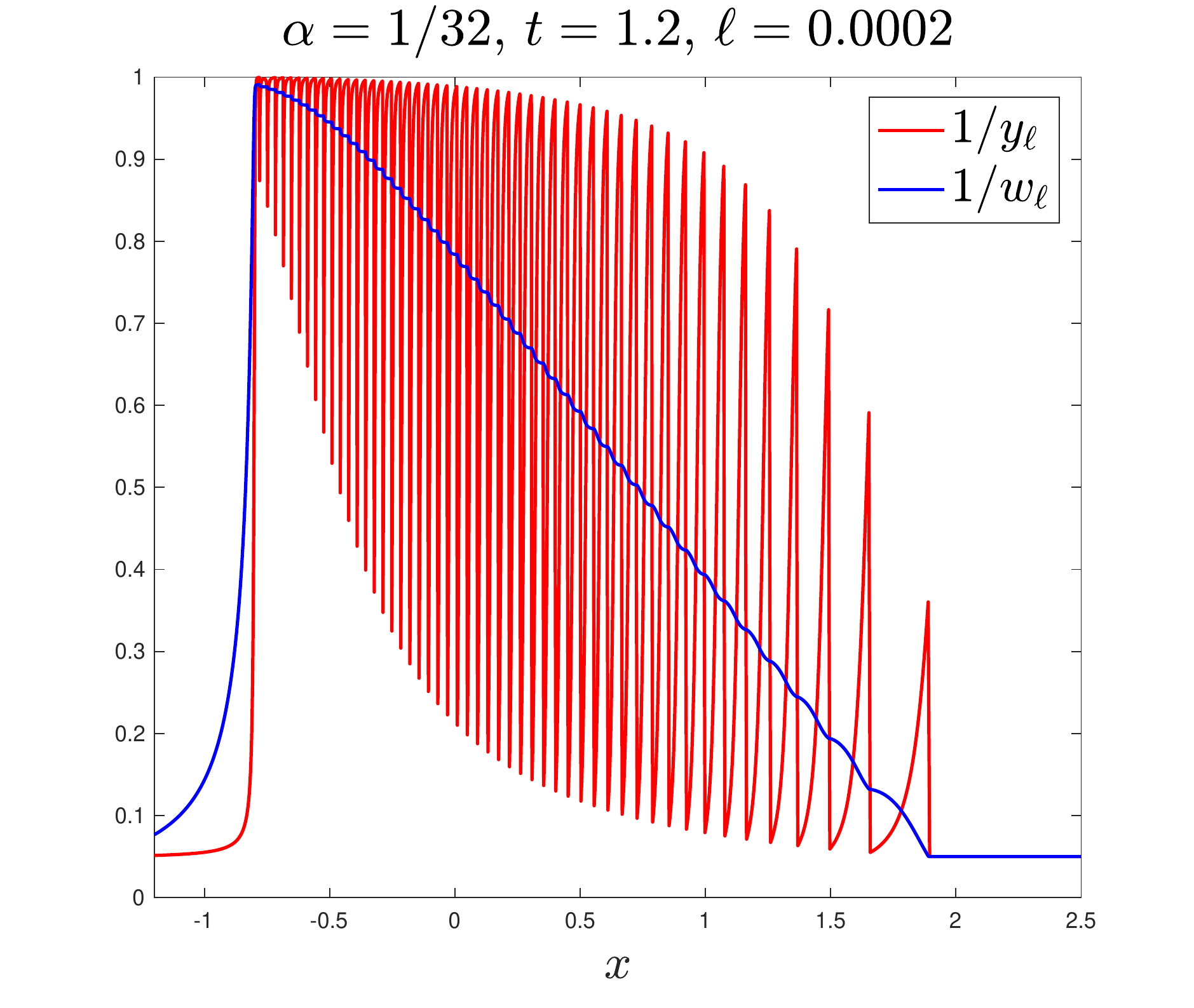}\\
    \includegraphics[width=0.5\linewidth]{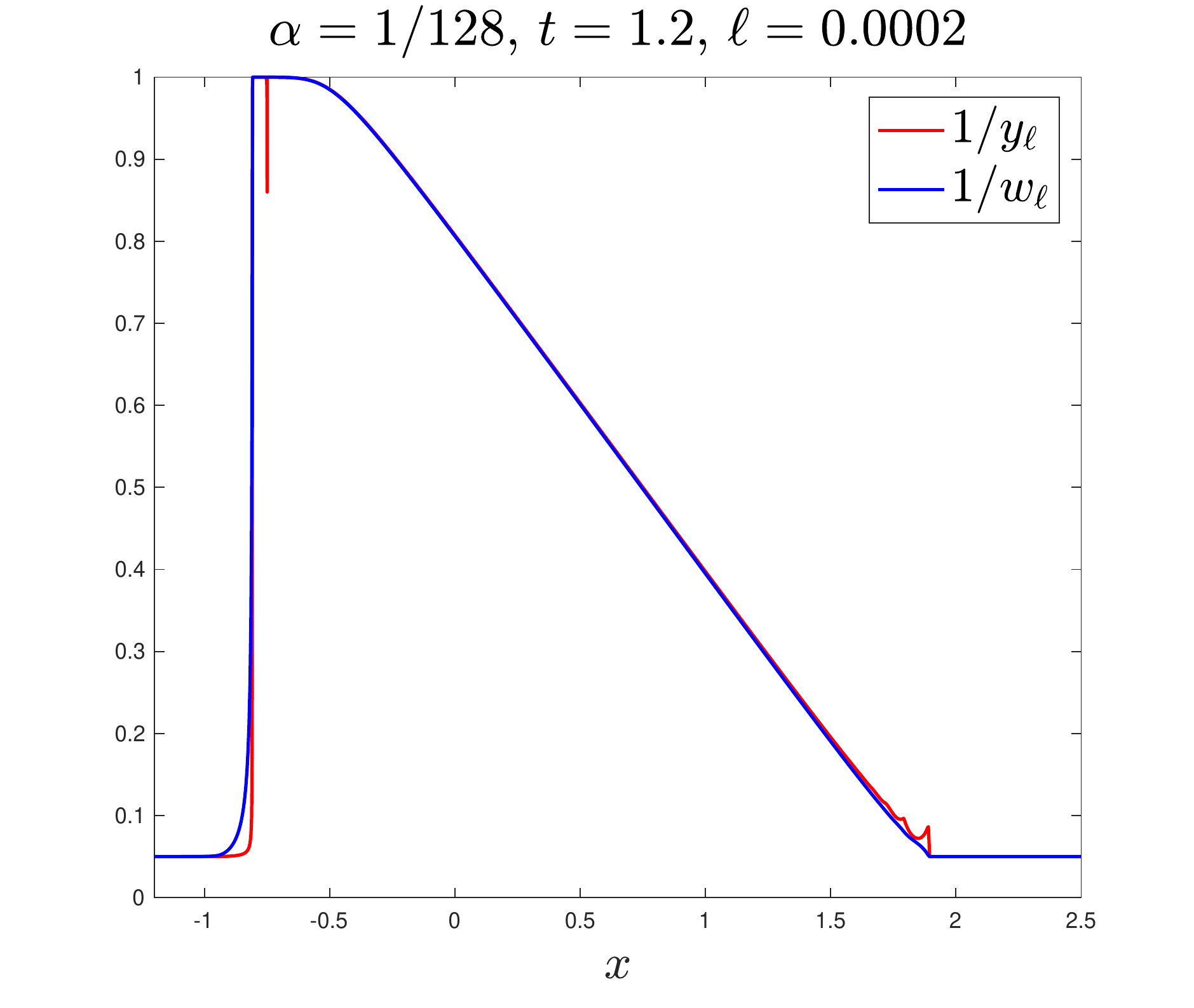}
    &\includegraphics[width=0.5\linewidth]{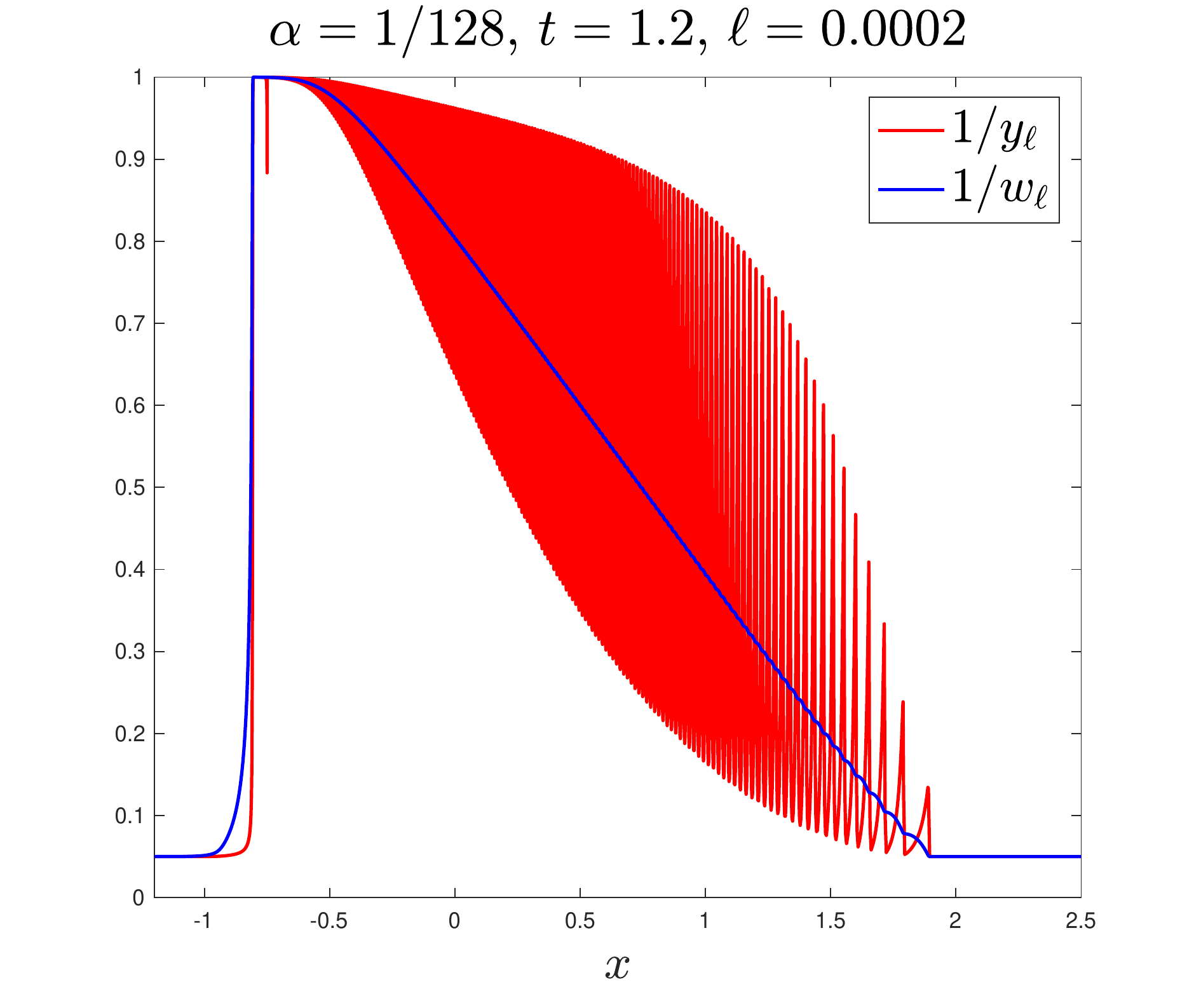}    
  \end{tabular}
  \vspace{-1.5em}
  \caption{Solutions of \eqref{eq:wschemep}, \eqref{eq:numbar},with
    initial data given in \eqref{eq:rho01}, \eqref{eq:initdata}. For
    all computations $t=1.2$ and $\ell=1/5000$. In the left column,
    $\Phi=\Phi_{\mathrm{tri}}$, in the right column, $\Phi=\Phi_{\mathrm{box}}$.}
  \label{fig:oscillations}
\end{figure}
From these computations, it is tempting to infer that (at least for
these initial data) $y_\ell$ converges strongly to $1/\rho$ for the
filter $\Phi_{\mathrm{tri}}$ and only weakly to $1/\rho$ for the
discontinuous filter $\Phi_{\mathrm{box}}$. To substantiate our
suspicion that $y_\ell$ only converges weakly, we did one final
experiment in which we used the same initial data, but $\ell=1/10000$
and $\alpha=1/256$.
\begin{figure}[h]
  \centering
  \begin{tabular}[h!]{lr}
    \includegraphics[width=0.5\linewidth]{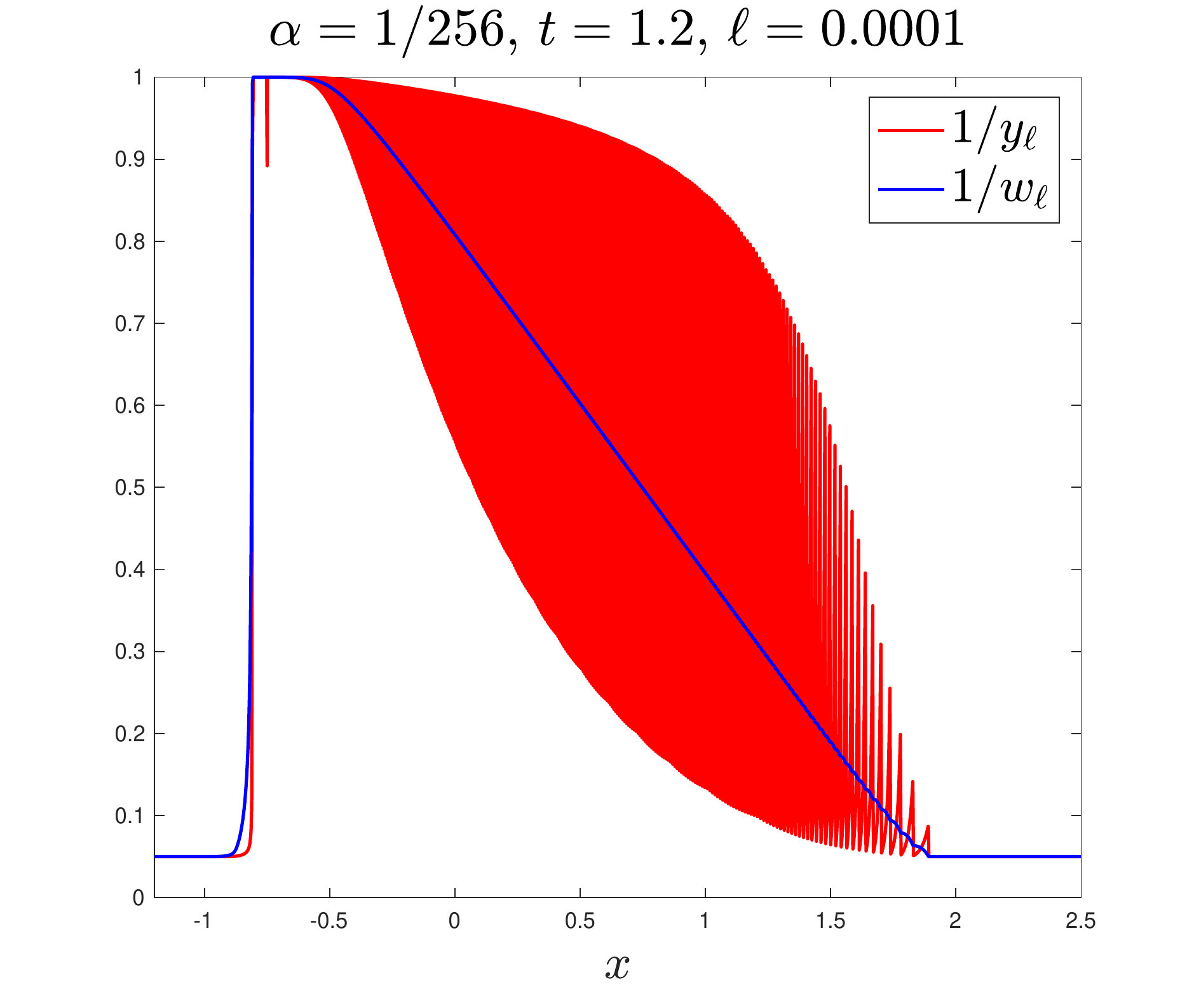}
    &\includegraphics[width=0.5\linewidth]{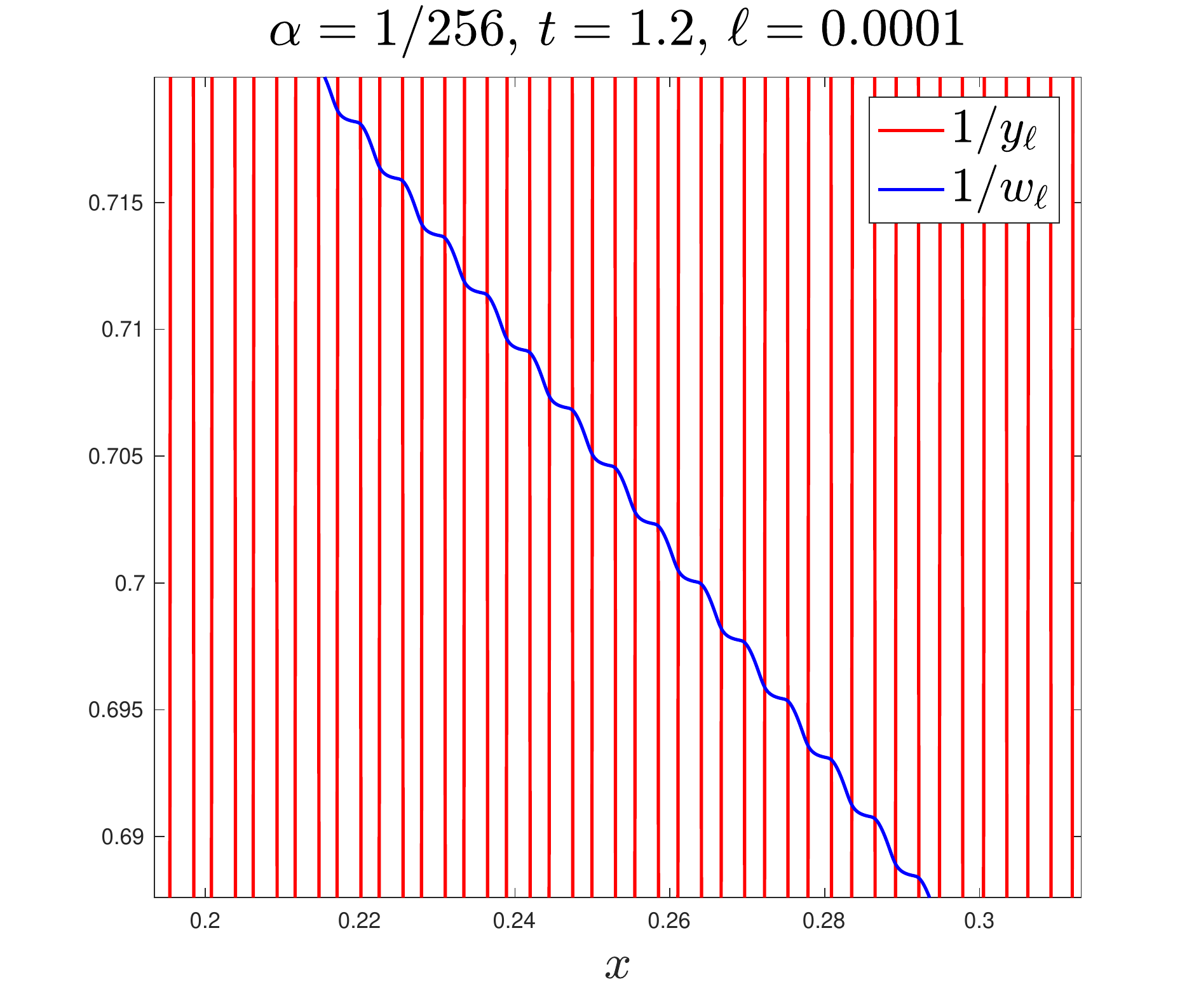}  
  \end{tabular}
  \vspace{-1.5em}
  \caption{Solutions of \eqref{eq:wschemep}, \eqref{eq:numbar},with
    initial data given in \eqref{eq:rho01}, \eqref{eq:initdata}, using
    the discontinuous filter $\Phi_{\mathrm{box}}$ with $\alpha=1/256$
    and $\ell=1/10000$. The figure to the right is just an enlargement of a
    region of the left figure.}
  \label{fig:ultrafine}
\end{figure}
The result is depicted in
Figure~\ref{fig:ultrafine}. The right figure is a magnification of the
region $x\in[0.2,0.3]$, $\rho\in [0.69,0.72]$ in the left figure.

Our experiment leads us to propose the conjecture 
that if a filter $\Phi$ is continuous, then 
the convergence of $y_\alpha$ to $1/\rho$ is strong.  
However, a proof has yet to be provided, except in the case 
of the exponential filter.



\end{document}